\documentclass[10pt,a4paper,final]{article}
\usepackage[utf8]{inputenc}
\usepackage{amsfonts}
\usepackage{amssymb}
\usepackage{fix-cm}
\usepackage{amscd,amssymb,stmaryrd}
\usepackage{amsmath}
\usepackage{graphicx}
\usepackage{subfigure}
\usepackage{fancybox}
\usepackage{amscd,amstext}
\usepackage{hyperref}
\usepackage{mathabx}
\usepackage{color}
\usepackage{epstopdf}
\usepackage{caption}
\usepackage{multicol}
\usepackage{cite}
\usepackage{amsthm}
\usepackage{mathtools}
\usepackage{bigints}
\usepackage{amsrefs}
\usepackage{showlabels} 
\usepackage{enumitem}
\usepackage{mathrsfs}
\usepackage{tikz}
\usepackage{empheq}
\usepackage{framed}
\usepackage{ulem}
\usepackage{nameref}
\usepackage{esint}
\numberwithin{equation}{section}
\usepackage{graphicx} 
\usepackage[margin=0.25in]{geometry}
\usepackage{pgfplots}
\pgfplotsset{width=11cm, compat=1.9}
\usepackage{tikz}

\newtheorem{refthm}{Theorem}

\newtheorem{theorem}{Theorem}
\newtheorem{lemma}{Lemma}

\newtheorem{corollary}{Corollary}

\newtheorem{proposition}{Proposition}

\newtheorem*{theorem*}{Theorem}
\newtheorem*{lemma*}{Lemma}
\newtheorem*{conj*}{Conjecture}
\newtheorem*{corollary*}{Corollary}
\newtheorem*{proposition*}{Proposition}
\newcommand{\rom}[1]{\uppercase\expandafter{\romannumeral #1\relax}}

\newcommand{\lan}{\langle}
\newcommand{\ran}{\rangle}

\newcommand{\R}{\mathbb{R}}
\newcommand{\C}{\mathbb{C}}

\newcommand{\N}{\mathbb{N}}

\newtheorem*{eg}{Example}

\DeclareMathOperator*{\sgn}{sgn}

\title{Restriction of Fractional Derivatives of the Fourier Transform }
\author{Michael Goldberg \and Chun Ho Lau}
\date{October 9, 2026}

\newcommand{\Addresses}{{
  \bigskip
  \footnotesize

  Michael GOLDBERG, \textsc{Department of Mathematical Science, University of Cincinnati, Cincinnati, OH 45221-0025, USA}\par\nopagebreak
  \textit{E-mail address:} \texttt{goldbeml@ucmail.uc.edu}

  \medskip

  Chun Ho LAU, \textsc{Department of Mathematical Science, University of Cincinnati, Cincinnati, OH 45221-0025, USA}\par\nopagebreak
  \textit{E-mail address:} \texttt{lauco@ucmail.uc.edu}
}}

\theoremstyle{remark}
\newtheorem*{remark}{Remark}

\begin{document}
\maketitle

\begin{abstract}
We consider the differentiability of the Fourier transform of $f \in L^p(\R^n)$ at points on the graph of a function with non-vanishing Gaussian curvature. In particular, we show that if the values of $\widehat{f}$ along the graph possess $H^\ell$ Sobolev regularity, then this implies a certain degree of regularity of $\widehat{f}$ in the last coordinate direction as well. Two types of results are proved: One controls the $\beta$-order derivative in the last direction in terms of its $H^{-s}$ norm for suitable combinations of $(\beta, p, s, \ell)$. This extends the result in \cite{GoldStol}*{Theorem 1.12} to fractional derivatives and to some of its sharp endpoints. In particular, we recover the sharp range in the $s=0$ case where $\frac{\partial \widehat{f}}{\partial \xi_n}$ has a surface restriction in $L^2$.  The second set of results is about $L^2(\R^n)$ bounds for difference quotients and fractional difference quotients of $\widehat{f}$, generalizing results in \cite{G2016}.
\end{abstract}
2020 Mathematics Subject Classification. 42B10; 42B20. 
\section{Introduction}
\subsection{Fourier restriction problem}
Properties of the Fourier transform of an $L^p(\R^n)$ function have long been a subject of study in classical harmonic analysis. The first result in this direction is Hausdorff-Young's inequality, which ensures that the Fourier transform of an $L^p(\R^n)$ function is in $L^{\frac{p}{p-1}}(\R^n)$ whenever $1\leq p \leq 2$. However, this does not give us any information if the Fourier transform of an $L^p(\R^n)$ function is restricted to a measure-zero subset $\Sigma\subset \R^n$. It is always possible to define restrictions when $p=1$ because the integral $\widehat{f}(\iota) := \int_{\R^n} e^{-2\pi i \iota \cdot x} f(x)dx$ converges for every $\iota \in \R^n$. When $p=2$, the Fourier transform is a surjective map onto $L^2(\R^n)$; thus, $\widehat{f}$ possesses only the properties of a generic $L^2$ function, whose values on a measure-zero set are undefined. 

For $1<p<2$, it is possible that $\widehat{f}$ has enough structure to be defined on certain measure-zero sets in $\R^n$.  The Fourier restriction problem, in the most general form, seeks to determine the complete range of $(p,q)$ for which
\begin{align}\label{FRPqp}
    \|\widehat{f}|_{\Sigma}\|_{L^q(\Sigma)}\lesssim \|f\|_{L^p(\R^n)}
\end{align}
holds for all $f\in L^p(\R^n)$ for a measure-zero set $\Sigma\subset \R^n$ with $n\geq 2$. Progress on this problem is closely related to the Kakeya conjecture, Bochner-Riesz summability, and partial differential equations; see \cites{Wolff, S}.

Different choices of $\Sigma$ and $n$ lead to different ranges of $(p,q)$, and it remains an open problem in most cases; the best result currently in $\R^3$ for parabolas is \cite{Wang1}; and for a higher dimension, see \cite{HR}. We also refer the reader to \cite{DO} for the survey on Fourier restriction estimates on different $\Sigma$. In particular, if $\Sigma$ is a two-dimensional sphere or a compact parabola around the origin, the range of $(p,q)$ is completely characterized by \cite{Feff}. It is also worth mentioning that the Tomas-Stein Theorem, established in \cites{tomas, Stein1}, states that Equation \eqref{FRPqp} holds for $q=2$ and $1\leq p\leq \frac{2n+2}{n+3}$ with $\Sigma$ having non-vanishing Gaussian curvature. 

One reason for requiring $\Sigma$ to have non-vanishing Gaussian curvature is to rule out ``uninteresting" cases. When $\Sigma$ is a hyperplane in $\R^{n}$, the only valid range of $(p,q)$ is $q=\infty$ and $p=1.$  Counterexamples can be constructed easily for all other $(p,q)$ combinations using separation of variables. In this paper, we will assume that $\Sigma$ is a graph of a smooth function $h$ defined on $\R^{n-1}$ with a vanishing gradient at 0 and an invertible Hessian matrix at 0. Our results are localized to a neighborhood $U$ of the origin where the Hessian matrix of $h$ continues to be invertible, using a cutoff function $\phi \in C^\infty_c(U)$. A precise statement will be given in Section \ref{notations}.

The structure question for Fourier transforms of $L^p(\R^n)$ functions we address in this paper is whether $\widehat{f}$ possesses sufficient regularity so that its derivatives (to some order) can be restricted to the surface $\Sigma$.  In \cite{GoldStol}, various versions of the derivative Fourier restriction problem are discussed, with the core idea being to impose suitable assumptions to obtain $\phi \nabla^k\widehat{f} \in H^{-s}(\Sigma)$. Here, $g\in H^{-s}(\Sigma)$ means $g|_{\Sigma}\in H^{-s}(\Sigma)$, and $H^{-s}(\Sigma)$ is the $L^2$-based Bessel potential space as defined in Section \ref{notations}. 

Cho, Guo, and Lee \cite{CGL}*{Theorem 1.1} studied the trace value of $\widehat{f}$ in Sobolev spaces for $f \in L^p(\R^n)$, including $p$ greater than the Tomas-Stein endpoint, with negative-order Sobolev norms. The first author and Stolyarov \cite{GoldStol}*{Proposition 1.1} used this to obtain a characterization of exponents for which $\|\phi \nabla^k\widehat{f}\|_{ H^{-s}(\Sigma)}\lesssim \|f\|_{L^p(\R^n)}$ holds.
\begin{refthm}[\cite{GoldStol}*{Proposition 1.1}] \label{prop1.1}
Let $p>1$. The inequality  $\|\phi \nabla^k\widehat{f}\|_{ H^{-s}(\Sigma)}\lesssim \|f\|_{L^p(\R^n)}$ holds for all $f\in L^p(\R^n)$ if and only if 
$k\leq s$, $k<\sigma_p$, and $2k-s\leq\kappa_p$, where $\kappa_p = \frac{n+1}{p}-\frac{n+3}{2}$ and $\sigma_p=\frac{n}{p}-\frac{n+1}{2}$.
\end{refthm}
More precisely, the Tomas-Stein theorem is the $(k, s) = (0,0)$ case of this statement, with the endpoint exponent occurring when $\kappa_p = 0$. The results of Cho, Guo, and Lee cover the remaining $k=0$ cases, and one can interpolate between them and a $(k, s, p) = (\frac{n-1}{2}, \frac{n-1}{2}, 1)$ endpoint (chosen because $\kappa_p = \sigma_p = \frac{n-1}{2}$) to obtain the rest.

The condition $k \leq s$ is natural in this context. Since the restriction of $\widehat{f}$ might belong to $L^2(\Sigma)$ without any additional smoothness, one might expect its $k^{\text{th}}$-order derivatives to belong to $H^{-k}(\Sigma)$ at best. If one further assumes that $\widehat{f} \in H^\ell(\Sigma)$, it immediately improves the regularity of partial derivatives tangent to $\Sigma$. The first author and Stolyarov showed that there is improvement in the regularity of partial derivatives in all directions, establishing some necessary conditions and some sufficient conditions for 
\begin{align} \label{HDR1}
    \|\phi \nabla^k \widehat{f}\|_{H^{-s}(\Sigma)}\lesssim \|f\|_{L^p(\R^n)}+\|\phi \widehat{f}\|_{H^{\ell}(\Sigma)}
\end{align}
for all $f\in L^p(\R^n)$ with $\phi \widehat{f} \in H^{\ell}(\Sigma)$.
For necessary conditions, they (\cite{GoldStol}*{Theorem 1.11}) showed that if \eqref{HDR1} holds, $p>1$, $s\geq 0$, and $\ell \geq 0$, then the exponents have to satisfy 
\begin{align}
    \label{N1} &k\leq s+1; \quad  \quad   k< \sigma_p; \quad   \quad  2k-s \leq \kappa_p;\\
   \label{N2} & k \leq \frac{\kappa_p(s+\ell)}{\kappa_p + \ell} \quad\text{ if  $k>s$ (and $\kappa_p > k>0$).} 
\end{align}
The range of admissible $(k,s)$ is larger than that of Theorem~\ref{prop1.1}, replacing the condition $s \geq k$ with the more permissive $s \geq k - \min(1, k, \min(1, \frac{\ell}{\kappa_p})(\kappa_p - k))$ whenever $0 \leq k \leq \kappa_p$.

Unfortunately, the proven sufficient conditions do not exactly match the necessary conditions due to needing the constraint $k\in \N$ at all intermediate steps of the argument. In order to have \eqref{HDR1}, other than \eqref{N1} and \eqref{N2}, they (\cite{GoldStol}*{Theorem 1.12}) additionally assume that $\kappa_p \in \N$ and either $2\lceil \frac{\ell-1}{\ell}\kappa_p\rceil \leq \kappa_p$ or $k\leq \lfloor \frac{\kappa_p}{2}\rfloor$ hold, where $\lceil x\rceil$ is the smallest integer greater than or equal to $x$ and $\lfloor x\rfloor$ is the greatest integer less than or equal to $x$. If $2\lceil \frac{\ell-1}{\ell}\kappa_p\rceil \leq \kappa_p$ and $k\leq \lfloor \frac{\kappa_p}{2}\rfloor$ do not hold, then one has to assume that $s\geq k-\frac{\kappa_p-k}{\kappa_p-\lfloor \kappa_p/2 \rfloor}$.

Nontrivial results are possible even with these extra limitations. In particular, if $k=1$ and $s=0$, 
\begin{refthm}[\cite{GoldStol}*{Corollary 1.13}] \label{GSC113}
    Let $m$ be an integer in $[2, \frac{n-1}{2})$ and $\ell\geq \frac{m}{m-1}$. Then, 
    $$\|\phi \nabla \widehat{f}\|_{L^2(\Sigma)}\lesssim \|f\|_{L^{\frac{2n+2}{n+3+2m}}(\R^n)}+\|\phi \widehat{f}\|_{H^{\ell}(\Sigma)}$$
    for all $f\in\mathcal{S}(\R^n)$ and $\phi\in C^{\infty}_c(\Sigma)$.  This range of $\ell$ is sharp with respect to the necessary conditions \eqref{N1} and \eqref{N2}.
\end{refthm}

The full derivative term $\|\phi \nabla^k \widehat{f}\|_{H^{-s}(\Sigma)}$ in \eqref{HDR1} is controlled by a combination of estimates on the transverse partial derivatives $\|\phi (\frac{\partial}{\partial \xi_n})^{k'} \widehat{f}\|_{H^{k-k'-s}(\Sigma)}$ for $0 \leq k' \leq k$. Thus Theorem~\ref{GSC113} and its extensions to $k\geq1$, $s \geq 0$ reduce to proving suitable bounds on $\phi (\frac{\partial}{\partial \xi_n})^{k'} \widehat{f}$.  The $k'=0$ case is true by assumption since $\ell \geq k-s$. It is a reasonable heuristic to conclude that~\eqref{HDR1} is equivalent to
\begin{align} \label{HDR3}
    \|\phi (\tfrac{\partial}{\partial \xi_n})^k \widehat{f}\|_{H^{-s}(\Sigma)}\lesssim \|f\|_{L^p(\R^n)}+\|\phi \widehat{f}\|_{H^{\ell}(\Sigma)}.
\end{align}
by applying some form of interpolation to handle the range $0 < k' < k$.

In fact, the interpolation in~\cite{GoldStol} between $k'=0$ and $k'=k$ takes place in integer steps using the Leibniz rule and algebraic identities. Continuous (e.g., complex analytic) interpolation is not used because the condition $\phi \widehat{f} \in H^\ell(\Sigma)$ may only be true for $f$ itself rather than for a family of functions containing $f$. That is what constrains $k$ to be an integer and permits sharp results only in the case where $\kappa_p = \frac{n+1}{p} - \frac{n+3}{2}$ is an integer as well.

In order to remove the integer constraints, we need to consider an alternative for $\frac{\partial}{\partial \xi_n}$ and to develop a framework more amenable to standard interpolation techniques. The Riesz fractional derivative, defined by $(-\Delta_{\xi_n})^{\beta/2}\widehat{f}(\xi)= [c_n|x_n|^\beta f]\widehat{\ }(\xi)$, is a natural choice to replace a differential operator. The question now becomes: 

\begin{description}
    \item[Question] For what ranges of $(\beta, p,s, \ell)$ does the inequality 
    \begin{equation*} 
    \|\phi (-\Delta_{\xi_n})^{\beta/2}\widehat{f}\|_{H^{-s}(\Sigma)}\lesssim \|f\|_{L^p(\R^n)}+\|\phi \widehat{f}\|_{H^{\ell}(\Sigma)}
    \end{equation*}
    hold true for $f\in\mathcal{S}(\R^n)$?
    In other words, if the Fourier transform of an $L^p(\R^n)$ function is $H^\ell$-smooth along the surface $\Sigma$, how much regularity does that imply in the transverse direction?
    
    This is essentially a Dirichlet-to-Neumann question, though we are unaware of a connection between these results and Dirichlet-to-Neumann phenomena in partial differential equations.
\end{description}

The new interpolation scheme will become evident in the statement of Theorems~\ref{main1} and~\ref{main2} below.

\subsection{Results and outline}
Our first main theorem provides a sufficient condition so that the fractional transverse partial derivative of $\widehat{f}$ is in $L^2(\Sigma)$.
\begin{theorem} \label{submain1}
    Suppose that $f\in L^p(\R^n)$ has $\widehat{f}|_{\Sigma}\in H^{\ell}(\Sigma)$, and $\phi\in C^{\infty}_c(\Sigma)$. If $p>1$, $ 2\beta \leq \kappa_p$, $0\leq \beta \leq 1$, and $\ell \geq \frac{\kappa_p\beta}{\kappa_p-\beta}$, then
\begin{enumerate}
        \item we have $(-\Delta_{\xi_n})^{\frac{\beta}{2}}\widehat{f} = c_n \widehat{|x_n|^{\beta}f}\in L^2(\Sigma)$ and 
$$\|\phi  (-\Delta_{\xi_n})^{\frac{\beta}{2}}\widehat{ f}\|_{L^2(\Sigma)}\lesssim_{n,p,\beta,\ell} \|f\|_{L^p(\R^n)} + \|\phi \widehat{f}\|_{H^{\ell}(\Sigma)}, $$
where $-\Delta_{\xi_n}$ denotes the Laplacian with respect to the $\xi_n$-variable.
\item if $\beta =1$ and $\ell \geq \frac{\kappa_p}{\kappa_p-1}$, we have $\frac{\partial}{\partial\xi_{n}}\widehat{f}\in L^2(\Sigma)$ and 
$$\bigg\|\phi  \frac{\partial \widehat{ f}}{\partial\xi_{n}}\bigg\|_{L^2(\Sigma)}\lesssim_{n,p,\ell} \|f\|_{L^p(\R^n)} + \|\phi \widehat{f}\|_{H^{\ell}(\Sigma)} $$

\end{enumerate}

When $p=1$, the result holds provided that the stricter condition $\ell > \frac{\kappa_1\beta}{\kappa_1-\beta}$ is assumed.
\end{theorem}
In particular, when $\beta=1$ we have a continuous extension of Theorem \ref{GSC113} which does not require $p$ to be in the form $\frac{2n+2}{n+3+2m}$ for an integer $m$. This result can be applied to $n=3$ and $n=4$, unlike Theorem \ref{GSC113}, because $\beta$ itself need not be an integer. From Theorem \ref{submain1}, the best possible range of $\beta$ is $[0,\frac{1}{2}]$ when $n=3$; and when $n=4$, the best range of $\beta$ is $[0, \frac{3}{4}]$. 

The second main theorem shows that the necessary conditions \eqref{N1} and \eqref{N2} are almost sufficient conditions for \eqref{HDR1}. 
\begin{theorem} \label{submain2}
    Let $s$, $\ell$, $\kappa_p$, $\sigma_p$, and $\beta$ be non-negative numbers. Fix $\phi\in C^{\infty}_c(\Sigma)$. Suppose $p>1$ and 
    \begin{align}
    \label{N1'} &\beta \leq s+1; \quad  \quad   \beta< \sigma_p; \quad   \quad  2\beta-s \leq \kappa_p.
\end{align}
When $s <\beta < \kappa_p$, define the quantity
\begin{equation} \label{tilde l}
\tilde{\ell} := \frac{\kappa_p(\beta - s)}{\kappa_p - \beta}.
\end{equation}
If $0 \leq \tilde{\ell} \leq 2$, then for every $\ell \geq \tilde{\ell}$, the inequality
\begin{align} \label{HDR4}
    \|\phi (-\Delta_{\xi_n})^{\beta/2}\widehat{f}\|_{H^{-s}(\Sigma)}\lesssim \|f\|_{L^p(\R^n)} + \|\phi \widehat{f}\|_{H^{\ell}(\R^n)}
\end{align}
holds for all $f\in \mathcal{S}(\R^n)$ such that $\widehat{f}|_{\Sigma}\in H^{\ell}(\Sigma)$. The condition $\ell \geq \tilde{\ell}$ is equivalent to the necessary condition $\frac{k\ell}{s+\ell-k} \leq \kappa_p$ stated in \eqref{N2}.

If $2 < \tilde{\ell} < \min( \frac{2\kappa_p}{2\kappa_p - \sigma_p}, \kappa_p)$, then \eqref{HDR4} holds for $\ell \geq \frac{\kappa_p}{\tilde{\ell}}(\tilde{\ell}-2) + 2$. It also holds in the case $2 < \tilde{\ell} = \kappa_p < \frac{2\kappa_p}{2\kappa_p - \sigma_p}$ and $\ell \geq \kappa_p = \tilde{\ell}$.

In the case with $\beta\in \N$, we also have 
\begin{align*}
    \bigg\|\phi\frac{\partial^{\beta}\widehat{f}}{\partial \xi_n^{\beta}}\bigg\|_{H^{-s}(\Sigma)}\lesssim \|f\|_{L^p(\R^n)} + \|\phi \widehat{f}\|_{H^{\ell}(\R^n)}.
\end{align*}

If $p=1$ and $0 \leq \tilde{\ell} \leq 2$, one condition can be relaxed to $\beta \leq \sigma_p$ while another is strengthened to $\ell > \tilde{\ell}$.
\end{theorem}

\begin{remark}
A comparison of the sufficient conditions in Theorem \ref{submain2} with the necessary conditions \eqref{N1} and \eqref{N2}. We say a pair $(\beta, s)$ is admissible if \eqref{N1'} is satisfied.
\begin{itemize}
    \item Cases with $s \geq \beta$ are covered by Theorem \ref{prop1.1}, where \eqref{N1} provides both necessary and sufficient conditions.  Admissible cases with $s < \beta$ only occur when $0 < \beta < \kappa_p$. 
    \item The necessary and sufficient conditions agree for all admissible pairs $(\beta, s)$ when $n \leq 5$, or $n>5$ and $p \in [\frac{2n+2}{n+7}, \frac{2n+2}{n+3}]$. 
    \item The necessary conditions are also sufficient when $n>5$, $p \in [1, \frac{2n+2}{n+7}]$ and $s \geq \beta - \frac{2(\kappa_p - \beta)}{\kappa_p}$, since this implies $\tilde{\ell} \leq 2$.  This includes all admissible cases with $\beta \in [0, \frac{\kappa_p}{2}].$
    \item In dimensions $n >5$, there is a positive result for all admissible pairs $(\beta, s)$ when $p \in (\frac{2n+4}{n+9}, \frac{2n+2}{n+7}]$, however the necessary condition $\ell \geq \tilde{\ell}$ is not achieved for admissible pairs with $\beta \in (\frac{\kappa_p}{2},\kappa_p)$ and $s < \beta - \frac{2(\kappa_p - \beta)}{\kappa_p}$ unless $s = 2\beta - \kappa_p$ exactly.
    \item In dimensions $n>5$ with $p \in [1, \frac{2n+4}{n+9}]$, the necessary condition $\ell \geq \tilde{\ell}$ is not achieved for admissible pairs with $\beta \in (\frac{\kappa_p}{2}, \kappa_p)$ and $s \in \big(\beta - \frac{2(\kappa_p - \beta)}{2\kappa_p - \sigma_p}, \beta - \frac{2(\kappa_p - \beta)}{\kappa_p}\big)$. Bounds of the form \eqref{HDR4} remain conjectural when $\beta \in (\frac{\kappa_p}{2}, \kappa_p)$ and $s \leq \beta - \frac{2(\kappa_p - \beta)}{2\kappa_p - \sigma_p}$.
\end{itemize}  
\end{remark}

While the results in Theorems \ref{submain1} and \ref{submain2} provide bounds on the derivatives of $\hat{f}$, they are largely derived by analysis of the associated difference quotient 
$$Tf(\xi,\eta) = \frac{\widehat{f}(\xi, h(\xi)+\eta) - \widehat{f}(\xi, h(\xi))}{\eta}, $$ 
$\xi\in U\subset \R^{n-1}$, and $\eta\in \R\setminus\lbrace 0\rbrace$. 

The difference quotient of Fourier restrictions is an object of independent interest. It is closely connected to the inhomogeneous Schr\"odinger propagator if $\Sigma$ is the paraboloid $\xi_n = \sum_{i \leq n-1} |\xi_i|^2$ and to the resolvent $(\Delta + 1)^{-1}$ if $\Sigma$ is the unit sphere.  The first author proved $L^p \to L^2$ estimates for $T$ under the assumption that $\widehat{f} \equiv 0$ on the sphere in \cite{G2016}. Here we extend the bound to cases where the restriction of $\widehat{f}$ to $\Sigma$ is nonzero but sufficiently smooth.

\begin{corollary} \label{coro1}
    Suppose $f\in L^p(\R^n)$ with $\widehat{f}|_{\Sigma}\in H^{\ell}(\Sigma)$, and $\phi\in C^{\infty}_c(\Sigma)$. If $1<p\leq \frac{2n+2}{n+5}$ and $\ell\geq \frac{\kappa_p}{2\kappa_p-1}$, then 
$$\|\phi Tf\|_{L^2_{\xi}L^2_{\eta}}\lesssim \|f\|_{L^p(\R^n)}+\|\phi \widehat{f}\|_{H^{\ell}(\Sigma)}.$$

The above estimate is also true for $p=1$ if $\ell > \frac{\kappa_1}{2\kappa_1-1}=\frac{n-1}{2n-4}.$
\end{corollary}
The smoothness condition $\ell  \geq \frac{\kappa_p}{2\kappa_p - 1}$ is sharp. It is necessary due to ``shifted Knapp" counterexamples identical to those in \cite{GoldStol}.

To derive the bounds for the derivatives of $\widehat{f}$, we consider an operator using $T$, namely,
\begin{align} \label{S_wf}
    S_wf(\xi)&= \int_{\R} w(r)\mathscr{F}_{\eta\rightarrow r}(Tf(\xi,\cdot))(-r)dr = \int_{\R} w(-r)\mathscr{F}_{\eta\rightarrow r}(Tf(\xi,\cdot))(r)dr 
\end{align}
The term $-r$ is to eliminate the negative sign obtained from the Fourier transform, and $\mathscr{F}_{\eta\rightarrow r}=\mathscr{F}_{\eta}$ denotes the one-dimensional Fourier transform from the $\eta$-variable to the $r$-variable. Note that $S_1f(\xi) = \lim_{\eta\to 0}Tf(\xi,\eta)$ is the restriction of $\frac{\partial \widehat{f}}{\partial \xi_n}$ to $\Sigma$.

With the operator $S_w$, we have the following theorems.
\begin{theorem} \label{main1}
       Suppose $f\in L^p(\R^n)$ has $\widehat{f}|_{\Sigma}\in H^{\ell}(\Sigma)$, and $\phi\in C^{\infty}_c(\Sigma)$. Under the hypothesis in Theorem \ref{submain1} with $\beta\neq 0$, for all $b\in\R$ we have
       $$  \sup_{\|w\|_{L^{\infty}}\leq 1}\|\phi S_{|u|^{\beta-1+ib}w}f\|_{L^2(\Sigma)} \lesssim \|f\|_{L^p(\R^n)}+\|\phi\widehat{f}\|_{H^{\ell}(\Sigma)}.$$
\end{theorem}
\begin{theorem} \label{main2}
Let $s$, $\ell$, $\kappa_p$, and $\sigma_p$ be non-negative numbers and $\beta >0$. Fix $\phi\in C^{\infty}_c(\Sigma)$. Under the same hypothesis in Theorem \ref{submain2},  for all $b\in\R$ we have
\begin{align} \label{HDR2}\sup_{\|w\|_{L^{\infty}}\leq 1}\|\phi S_{|u|^{\beta-1+ib}w}f\|_{H^{-s}(\Sigma)} \lesssim \|f\|_{L^p(\R^n)}+\|\phi\widehat{f}\|_{H^{\ell}(\Sigma)}
\end{align} 
whenever  $f\in L^p(\R^n)$ with $\widehat{f}|_{\Sigma}\in H^{\ell}(\Sigma)$.
\end{theorem}

We remark that the implicit constants depend on $s$, $\ell$, $\kappa_p$, $\sigma_p$, and $\beta$, but they are independent of $b$. Moreover, it suffices to consider $f\in \mathcal{S}(\R^n)$, the space of Schwartz functions on $\R^n$, thanks to a density lemma \cite{GoldStol}*{Lemma 3.10}.

Theorems \ref{submain1} and \ref{submain2} follow from Theorems \ref{main1} and \ref{main2} by choosing $w=i\sgn(u)$ in the general case and $w=i\sgn(u)^{\beta-1}$ when $\beta \in \N$. This is because one can show that
\begin{align*} 
&S_{i\sgn(u)|u|^{\beta-1}}f(\xi)\nonumber\\
&= -2 \pi^{3/2}\beta^{-1} \int_{\R^n} |x_n|^{\beta}e^{-2\pi i (x'\cdot \xi+x_nh(\xi))}f(x',x_n)dx'dx_n \simeq [(-\Delta_{\xi_n})^{\frac{\beta}{2}}\widehat{f}]|_{\Sigma}(\xi)
\end{align*}
and similarly one has $S_{i\sgn(u)^{\beta-1}|u|^{\beta-1}}f \simeq [x_n^{\beta}f]\widehat{\ }|_{\Sigma}(\xi)$ if $\beta \in \N$. From this calculation, we have to exclude the case $\beta =0$ in Theorem \ref{main1}, but we can include $\beta =0$ in Theorem \ref{submain1} because it follows from the hypothesis.

The remainder of this paper is organized as follows. In Section \ref{Pre}, we state the assumptions on $h$, define different types of Sobolev spaces, and compute the kernel of $S_w$. We gather some useful estimates from \cite{GoldStol} and prove our first main lemma. In Section \ref{mainproof}, the proofs of Theorems \ref{main1} and \ref{main2} are shown. The proof of Corollary \ref{coro1} is given in Section \ref{proofcoro}. In Section \ref{OE}, we consider a variation of Corollary \ref{coro1} with fractional order difference quotients.

\section{Definitions and building blocks} \label{Pre}
\subsection{Notations} \label{notations}
In this paper, we assume that $\Sigma$ is a graph of a function $h$ defined on $\R^{n-1}$.  Let $U$ be a neighborhood of the origin in $\R^{n-1}$. We assume that the function $h \in C^{\infty}(U)$ such that $h(0)=0$, $\nabla h(0)=0$,  $\det(\frac{\partial^2 h}{\partial \xi^2}(0))\neq 0$, and the second differential of $h$ in $U$ remains invertible at every $\xi \in U$.  Given a value $\eta \in \R$, the pair $(\xi, \eta)$ defines a unique point $(\xi, h(\xi) + \eta) \in \R^n$.

The inner product of the homogeneous negative Sobolev space on $\Sigma$,  $\dot{H}^{-\gamma}(\Sigma)$, for $\gamma\in (0,\frac{n-1}{2})$ is given by 
$$\langle \phi, \psi\rangle_{\dot{H}^{-\gamma}(\Sigma)} :=C_{n,\gamma} \int_{U\times U} \phi(\xi, h(\xi))\overline{\psi(\zeta, h(\zeta))}|\xi-\zeta|^{2\gamma-n+1}d\xi d\zeta.$$ 
For $\ell\in \R$, the norm of the inhomogeneous Sobolev space on $\Sigma$, $H^{\ell}(\Sigma)$, is 
$$\|\psi\|_{H^{\ell}(\Sigma)}^2:= \int_{\R^{n-1}}\big|\mathscr{F}_{\xi\rightarrow z}[\psi(\xi, h(\xi))](z)\big|^2(1+|z|)^{2\ell}dz,$$
where $\mathscr{F}_{\xi\rightarrow z}$ denotes the usual Fourier transform from $\xi$-variable to $z$-variable. Since $U$ is a neighborhood of the origin, the homogeneous and inhomogeneous norms are equivalent for $\ell=-\gamma\in (0,\frac{n-1}{2})$.

In this paper, the implicit constant in $\lesssim$ may vary from line to line.

Before we move on to calculating the kernel of $S^*_wS_w$, we shall recall the notations of some useful exponents in this paper:
\begin{align*} 
   \sigma_p &:= \frac{n}{p} -\frac{n+1}{2};\\
    \kappa_p &:= \frac{n+1}{p} -\frac{n+3}{2}.
\end{align*}
\subsection{Computing the kernel of $S^*_wS_w$}
Let us recall $S_w$, where $w$ is a complex-valued bounded measurable function. 
\begin{align*} 
    S_wf(\xi)&= \int_{\R} w(r)\mathscr{F}_{\eta\rightarrow r}(Tf(\xi,\cdot))(-r)dr = -\int_{\R} w(-r)\mathscr{F}_{\eta\rightarrow r}(Tf(\xi,\cdot))(r)dr 
\end{align*}

We first compute the Fourier transform of $T$. To do so, for $f\in \mathcal{S}(\R^n)$, 
\begin{align*}
     &\mathscr{F}_{\eta\to r}(Tf(\xi, \cdot))( -r)\\ &= \lim_{\varepsilon\rightarrow 0^+} \int_{-\infty}^{\infty} \int_{\R^n} e^{-\pi^2\varepsilon \eta^2} e^{2\pi i\eta r}  \bigg(\frac{e^{-2\pi i x_n \eta}-1}{\eta}\bigg)  e^{-2\pi i (x'\cdot \xi + x_n h(\xi))} f(x',x_n)dx'dx_n d\eta \\
    &= \lim_{\varepsilon\rightarrow 0^+} -2\pi i \int_{\R^n} E_{x_n,\varepsilon}(r)  e^{-2\pi i (x'\cdot \xi + x_n h(\xi))} f(x',x_n)dx'dx_n,
\end{align*}
where 
$$E_{x_n,\varepsilon}(r):= \frac{1}{\sqrt{\varepsilon}}\int_{-x_n}^{0} e^{-\frac{(-r-y)^2}{\varepsilon}}dy= \int_{\frac{r-x_n}{\sqrt{\varepsilon}}}^{\frac{r}{\sqrt{\varepsilon}}}e^{-t^2}dt.$$
Note that as $\varepsilon \rightarrow 0^+$, we have 
$$s_{r}(x_n)=\sqrt{\pi}1_{x_n}(r):=\lim_{\varepsilon\rightarrow 0^+} E_{x_n,\varepsilon}(r)=\begin{cases}
    \sqrt{\pi}, &\quad \text{if }0<r<x_n;\\
    -\sqrt{\pi}, &\quad \text{if }x_n<r<0;\\
    -\sgn(x_n)\frac{\sqrt{\pi}}{2} &\quad \text{if } r=0;\\
    0, &\quad \text{otherwise}.
\end{cases}$$
Therefore, using the Dominated Convergence Theorem, we have 
\begin{align*}
    \mathscr{F}_{\eta\to r}(Tf(\xi, \cdot))( -r) &=   -2\pi i \int_{\R^n} s_r(x_n)  e^{-2\pi i (x'\cdot \xi + x_n h(\xi))} f(x',x_n)dx'dx_n.
\end{align*}

This gives an integral kernel for $S_w$ as
\[
S_wf(\xi) = -2\pi i \int_{\R^n}W(x_n) e^{-2\pi i(x'\cdot\xi+x_nh(\xi))}f(x',x_n)dx'dx,
\]
where $W(x_n) := \int_{-\infty}^\infty w(r)s_r(x_n) dr$.

It is now a straightforward matter to compute the kernel of $S_w^*S_w$, where $S_w$ is considered as a map from $\mathcal{S}(\R^n)$ to $\dot{H}^{-\gamma}(\Sigma)$. Write the inner product
\begin{align}
    \lan \phi S_w f,\phi S_w g\ran_{\dot{H}^{-\gamma}(\Sigma)}
    = 4\pi^2 \int_{\R^n \times \R^n} W(x_n)\overline{W(y_n)} \widecheck{\phi^2}_{-\gamma, \Sigma}(x,y) f(x)\overline{g(y)} dx dy,
\end{align}
where
\[
\widecheck{\phi^2}_{-\gamma,\Sigma}(x,y) := \lan \phi e^{-2\pi i(x \cdot \xi)}, \phi e^{-2\pi i(y \cdot \xi)} \ran_{\dot{H}^{-\gamma}(\Sigma)}.
\]

That makes the kernel of $(\phi S_w)^*(\phi S_w)$ equal to
\begin{equation} \label{S*S kernel}
K_w(x,y) = \widecheck{\phi^2}_{-\gamma,\Sigma}(x,y) W(x_n)\overline{W(y_n)}.
\end{equation}
In the calculations that follow, we will take advantage of an operator bound on $\widecheck{\phi^2}_{-\gamma,\Sigma}(x,y)$ from~\cite{GoldStol}*{inequality (61)}:
\begin{proposition} \label{prop:bilinear}
    Let $0 \leq \gamma < \frac{n-1}{2}$ and $p \in [1, \frac{2n-2}{n-1+2\gamma}]$. Given functions $f, g \in L^p(\R^{n-1})$, and with $\widecheck{\phi^2}_{-\gamma, \Sigma}(x,y)$ as defined above,
    \begin{multline} \label{bilinear}
        \Big|\iint_{\R^{n-1}\times\R^{n-1}} f(x') \overline{g(y')}\widecheck{\phi^2}_{-\gamma,\Sigma}(x',x_n,y',y_n) dx' dy'\Big| \\
        \lesssim (1+|x_n|)^{-\gamma} (1+|y_n|)^{-\gamma} (1 + |x_n-y_n|)^{\gamma - \frac{n-1}{p} + \frac{n-1}{2}} \|f\|_{L^p(\R^{n-1})} \|g\|_{L^p(\R^{n-1})}. 
    \end{multline}
\end{proposition}

We will also use the following version of a Stein-Weiss inequality,  proved as Theorem 7.4 in \cite{GoldStol}.
\begin{lemma} \label{SteinWeiss}
    Let $f, g \in L^p(\R)$, $1 \leq p < 2$. Then
    \[
    \iint_{\R \times \R} (1+|x_n|)^{-a}(1+|y_n|)^{-a}(1+|x_n-y_n|)^{-b}f(x_n)g(y_n) dx dy \lesssim \|f\|_p \|g\|_p
    \]
    if and only if $a$ and $b$ satisfy
    \begin{align*}
        &0 \leq a < 1 - \tfrac{1}{p} \text{ \ and \ } b \geq 2(1-\tfrac{1}{p} - a), \text{ or} \\
        &1-\tfrac{1}{p} \leq a \text{ \ and \ } b > 1-\tfrac{1}{p} - a
    \end{align*}
    for $p>1$, and with the endpoint $b = - a$ included if $p=1$.
\end{lemma}

\subsection{A modified Cauchy-Schwarz lemma}
One key step in the proof of Theorem \ref{main1} is to place the desired bound in between an estimate with no $\xi_n$-derivatives and good regularity in $\xi$  (i.e., $\phi\widehat{f} \in H^\ell(\Sigma)$) and an estimate with twice as many $\xi_n$-derivatives but lower regularity.  We note that the pairing of different Sobolev spaces $H^s(\Sigma)$ is governed by a Cauchy-Schwarz inequality in the midpoint space. The pairing of different numbers of derivatives in $\xi_n$ is more complicated, but in some cases it resembles a duality pairing with an additional remainder term.

\begin{lemma} \label{iterate1}
    Let $a\in (-1,\infty)$, $0\leq A < \frac{n-1}{2}$, $\ell\geq 0$, and $w\in L^{\infty}(\R;\C)$ with norm at most 1. We have 
\begin{equation} \label{eqn:iterate1}\|\phi S_{|u|^{a+ib}w}f\|_{\dot{H}^{-A}(\Sigma)}^2\lesssim \|\phi\widehat{f}\|_{H^{\ell}(\Sigma)}\sup_{\|w\|_{L^{\infty}}\leq 1} \|\phi S_{|u|^{2a+1+bi}w}f\|_{\dot{H}^{-2A-\ell}(\Sigma)}+\|f\|_{L^p(\R^n)}^2
\end{equation}
provided that $\kappa_p\geq 2a-A+2$, $\sigma_p>2a+2$, $A\geq a$, and everything on the right-hand side is finite.
\end{lemma}
\begin{proof}
We first observe that for complex-valued $w=w_R+iw_I$, where $w_R$ and $w_I$ are real-valued bounded measurable functions, we have
\begin{align*}
    S_{|u|^{a+ib}w}f(\xi) &= \int_{\R} |r|^{a+ib}[w_R(r)+iw_I(r)]\mathscr{F}_{\eta\rightarrow r}(Tf(\xi,\dot))(-r)dr \\
    &=S_{|u|^{a+ib}w_R}f(\xi) + i S_{|u|^{a+ib}w_I}f(\xi).
\end{align*}

\begin{align*}
    \sup_{\|w\|_{L^{\infty}(\R;\C)}\leq 1}\|\phi S_{|u|^{a+ib}w}f\|_{\dot{H}^{-A}(\Sigma)}\leq 2\sup_{\|w'\|_{L^{\infty}(\R;\R)}\leq 1}\|\phi S_{|u|^{a+ib}w'}f\|_{\dot{H}^{-A}(\Sigma)}.
\end{align*}

Moreover, $L^{\infty}(\R;\R)\subset L^{\infty}(\R;\C)$, so we can focus on $w\in L^{\infty}(\R;\R)$ with $\|w\|_{L^{\infty}(\R;\R)}\leq 1$.

Consider the quantity on the left side of \eqref{eqn:iterate1}.
    \begin{align*}
        &\langle \phi S_{|u|^{a+ib}w}f, \phi S_{|u|^{a+ib}w}f\rangle_{\dot{H}^{-A}(\Sigma)} \\
        &\simeq \iint_{\R^n\times \R^n} \widecheck{\phi^2}_{-A,\Sigma}(x,y)\bigg(\frac{|W_{a,b}(x_n)|^2+|W_{a,b}(y_n)|^2}{2} - \frac12 R(x_n,y_n) \bigg)f(x',x_n)\overline{f(y',y_n)} dydx \\
        &=: (I)+(II)-(III),
    \end{align*}
where 
$$W_{a,b}(x_n):= \int_{\R} |r|^{a+ib}w(r)\sqrt{\pi}1_{x_n}(r)dr,$$
and $R(x_n,y_n)= |W_{a,b}(x_n)-W_{a,b}(y_n)|^2.$

The terms $|(I)|$ and $|(II)|$ can be controlled similarly, so it suffices to consider $|(I)|.$ By writing 
$$(I)= \bigg\langle \phi \mathscr{F}\bigg(\frac{|W_{a,b}(x_n)|^2}{2}f\bigg), \phi \widehat{f}\bigg\rangle_{\dot{H}^{-A}(\Sigma)},$$
one has 
\begin{align*}
    |(I)|&\leq \bigg\| \phi \mathscr{F}\bigg(\frac{|W_{a,b}(x_n)|^2}{2}f\bigg)\bigg\|_{\dot{H}^{-2A-\ell}(\Sigma)}\|\phi\widehat{f}\|_{\dot{H}^{\ell}(\Sigma)} \leq \bigg\| \phi \mathscr{F}\bigg(\frac{|W_{a,b}(x_n)|^2}{2}f\bigg)\bigg\|_{\dot{H}^{-2A-\ell}(\Sigma)}\|\phi\widehat{f}\|_{H^{\ell}(\Sigma)}.
\end{align*}

If we define
$$\widetilde{w}(u):= |u|^{-a-1}w(u)W_{a,b}(u)+|u|^{-a-1-2bi}w(u)W_{a,b}(u),$$
then $\int_{\R}|u|^{2a+1+bi}\widetilde{w}(u)\sqrt{\pi}1_{x_n}(u)du =\frac{1}{2}|W_{a,b}(x_n)|^2$ and $\|\widetilde{w}\|_{L^{\infty}(\R)}\lesssim 1$ because $|W_{a,b}(u)|\lesssim |u|^{a+1}$ and $\|w\|_{L^{\infty}}\leq 1.$ 
Therefore, 
\begin{align*}
    \bigg\| \phi \mathscr{F}\bigg(\frac{|W_{a,b}(x_n)|^2}{2}f\bigg)\bigg\|_{\dot{H}^{-2A-\ell}(\Sigma)} &\simeq \|\phi S_{|u|^{2a+1+bi}\widetilde{w}}f\|_{\dot{H}^{-2A-\ell}(\Sigma)} \lesssim \sup_{\|w\|_{L^{\infty}}\leq 1} \|\phi S_{|u|^{2a+1+
bi}w}f\|_{\dot{H}^{-2A-\ell}(\Sigma)}.
\end{align*}

The condition $a\in (-1,\infty)$ was used to obtain an estimate for $|W_{a,b}(u)|$. The relation between $p$, $a$, and $A$ will be used in the next steps.

To estimate $|(III)|$, first rewrite it with attention to whether coordinates $x_n$ and $y_n$ are of similar size.
\begin{align*}
    |(III)|&\lesssim \Big|\bigg(\int_{|x_n|\geq 2|y_n|}+\int_{\frac{|x_n|}{2}\leq |y_n|\leq 2|x_n|} +\int_{|y_n|\geq 2|x_n|}\bigg) |R(x_n, y_n)| \widecheck{\phi^2}_{-A,\Sigma}(x,y) f(x',x_n) \overline{f(y',y_n)}dydx\Big|\\
    &=: (III_1)+(III_2)+(III_3).
\end{align*}
We have two estimates on the size of $R(x_n, y_n)$. Applying the mean value theorem to the difference $W_{a,b}(x_n) - W_{a,b}(y_n)$ gives a bound $|R(x_n,y_n)|\lesssim |x_n-y_n|^2 \max(|x_n|, |y_n|)^{2a}$. Ignoring cancellation entirely gives a bound of $|R(x_n, y_n)| \lesssim \max(|x_n|, |y_n|)^{2a+2}$. In combination, this yields the composite bound
\[|R(x_n, y_n)| \lesssim 
\begin{cases}
(1+|x_n|)^{\max(0,a)}(1+|y_n|)^{\max(0,a)}(1+|x_n-y_n|)^{\min(2a+2,2)}, \text{ \ when } |x_n| \sim |y_n|, \\
(1+|x_n-y_n|)^{2a+2}, \text{ \ when } |x_n| \geq 2|y_n| \text{ or } 2|x_n| \leq |y_n|.
\end{cases}
\]

For $(III_1)$, this combines with \eqref{bilinear} and Lemma \ref{SteinWeiss} to give the estimate
\begin{align*}
|(III_1)| &\lesssim \iint_{\R\times\R} \begin{aligned}[t](1+|x_n|)^{-A}&(1+|y_n|)^{-A}(1+|x_n-y_n|)^{2a+2+A-\frac{n-1}{p}+\frac{n-1}{2}} \\
&\times\|f(\,\cdot\,,x_n)\|_{L^p(\R^{n-1})} \|f(\,\cdot\,,y_n)\|_{L^p(\R^{n-1})} dx_n dy_n
\end{aligned} \\
& \lesssim \|f\|_{L^p(\R^n)}^2
\end{align*}
provided $A\geq 0$, $2a+2-A \leq \kappa_p$, and $2a+2 < \sigma_p$ (equality is also permitted when $p=1$). The bound for $(III_3)$ is identical.

For $(III_2)$, it combines with \eqref{bilinear} and Lemma \ref{SteinWeiss} to give the estimate
\begin{align*}
|(III_2)| &\lesssim \iint_{\R\times\R} \begin{aligned}[t](1+|x_n|)^{\max(0,a)-A}(1+|y_n|)^{\max(0,a)-A}(1+|x_n-y_n|)^{\min(2a+2, 2)+A-\frac{n-1}{p}+\frac{n-1}{2}}& \\
\times\|f(\,\cdot\,,x_n)\|_{L^p(\R^{n-1})} \|f(\,\cdot\,,y_n)\|_{L^p(\R^{n-1})} dx_n dy_n&
\end{aligned} \\
& \lesssim \|f\|_{L^p(\R^n)}^2
\end{align*}
provided $A \geq \max(0,a)$, $2a+2-A \leq \kappa_p$, and $2a+2 - \max(0,a) < \sigma_p$ (with equality permitted when $p=1$).
\end{proof}

\section{Proof of Theorem \ref{main1} and \ref{main2}} \label{mainproof}
\subsection{Proof of Theorem \ref{main1}}
\subsubsection{Case $p>1$}
Assume $\beta\in (0,1]$. By applying Lemma \ref{iterate1} to $A=0\geq a=\beta -1 >-1$, we have
\begin{align} \label{beta/2}
    \|\phi S_{|u|^{\beta-1+ib}w}f\|_{L^2(\Sigma)}\lesssim \|\phi\widehat{f}\|_{H^{\ell}(\Sigma)}^{\frac12}\sup_{\|w\|_{L^{\infty}}\leq 1}\|\phi S_{|u|^{2\beta -1+ib}w}\|^{\frac12}_{\dot{H}^{-\ell}(\Sigma)}+\|f\|_{L^p(\R^n)}
\end{align}
with $\kappa_p\geq 2\beta$ and $\sigma_p>2\beta$.

We also need an \textit{a priori} estimate.
\begin{lemma} \label{apriori1}
For $1<p\leq \frac{2n+2}{n+3}$, and $\kappa_p-1 \leq a < \sigma_p-1$, we have
    $$\sup_{\|w\|_{L^{\infty}(\R;\C)}\leq 1, b\in \R}\|\phi S_{|u|^{a+ib}w(u)}f\|_{\dot{H}^{-(2a+2 - \kappa_p)}(\Sigma)}\lesssim \|f\|_{L^p(\R^n)}.$$
\end{lemma}
\begin{proof}
Let us write $W_{a,b}(x_n):=\int_{\R}|r|^{a+ib}w(r)\sqrt{\pi}1_{x_n}(r)dr$ as before. 

Since $\|w\|_{L^{\infty}}\leq 1$, we have
\begin{align*}
    &|W_{a,b}(x_n)| = \bigg|\int_{\R}|r|^{a+ib}w(r)\sqrt{\pi}1_{x_n}(r)dr\bigg|\lesssim  |x_n|^{a+1}.
\end{align*}

The kernel formula for $(\phi S_w)^*(\phi S_w)$ given in \eqref{S*S kernel} can be applied to $|u|^{a+ib}w(u)$ to obtain
\begin{align*}
    \|\phi S_{|u|^{a+ib}w(u)}f &\|_{\dot{H}^{-(2a+2-\kappa_p)}(\Sigma)}^2 \\
    &= \iint_{\R^n\times\R^n}\widecheck{\phi^2}_{-(2a+2-\kappa_p), \Sigma}(x,y) W_{a,b}(x_n)\overline{W_{a,b}(y_n)} f(x)\overline{f(y)} dx dy \\
    &\lesssim 
    \begin{aligned}[t]
        \iint_{\R\times \R} (1+|x_n|)^{-(2a+2-\kappa_p)}(1+|y_n|)^{-(2a+2-\kappa_p)}(1+|x_n-y_n|)^{2a -\frac{2n}{p} + \frac{2n+6}{2}} &\\
        \times |W_{a,b}(x_n)W_{a,b}(y_n)|
        \|f(\,\cdot\,,x_n)\|_{L^p(\R^{n-1})} \|f(\,\cdot\,,y_n)\|_{L^p(\R^{n-1})} dx_n dy_n
    \end{aligned} \\
    &\lesssim 
    \begin{aligned}[t]
        \iint_{\R\times \R} (1+|x_n|)^{-(a+1-\kappa_p)} (1+|y_n|)^{-(a+1-\kappa_p)}(1+|x_n-y_n|)^{2a - \frac{2n}{p} + \frac{2n+6}{2}} &\\
        \times \|f(\,\cdot\,,x_n)\|_{L^p(\R^{n-1})} \|f(\,\cdot\,,y_n)\|_{L^p(\R^{n-1})} dx_n dy_n&
    \end{aligned} \\
    &\lesssim \|f\|_{L^p(\R^n)}^2.
\end{align*}
The final inequality in the calculation above is precisely the $0 \leq a <1-\frac1p$, $b = 2(1-\frac{1}{p}-a)$ case of Lemma \ref{SteinWeiss}.
\end{proof}

Now we return to \eqref{beta/2} and control the $\|\phi S_{|u|^{2\beta-1+ib}w}\|_{\dot{H}^{-\ell}(\Sigma)}$ norm by interpolation. Choose $\tilde{\ell} := \frac{\kappa_p \beta}{\kappa_p - \beta}$ so that the three points $(0,-\tilde{\ell})$, $(\beta, 0)$, and $(\kappa_p, \kappa_p)$ in the plane are collinear.  The affine combination $(2\beta, \tilde{\ell})$ is also collinear.

Using Linder\"olf Theorem by regarding $(1+|\xi|)^{m+ib}$ as a weight and applying to the function $\widehat{f}$, set $\theta\in [0,1]$ such that $2\beta-1=(1-\theta)(\kappa_p-1)+\theta(\beta-1)$, \textit{i.e.} $\theta = \frac{\kappa_p-2\beta}{\kappa_p-\beta}$. then $(1-\theta)\kappa_p = \frac{\kappa_p \beta}{\kappa_p - \beta} = \tilde{\ell}$, so we have 
\begin{align}\label{itereq}
    \|\phi S_{|u|^{2\beta -1+ib}w}f\|_{\dot{H}^{-\tilde{\ell}}(\Sigma)}&\lesssim \sup_{\|w\|_{L^{\infty}}\leq 1}\|\phi S_{|u|^{\beta-1+ib}w}f\|_{L^2(\Sigma)}^{\theta}  \sup_{\|w\|_{L^{\infty}}\leq 1}\|\phi S_{|u|^{\kappa_p-1+ib}w}f\|_{\dot{H}^{-\kappa_p}(\Sigma)}^{1-\theta}\nonumber\\
    &\lesssim \big(\|\phi\widehat{f}\|_{H^{\tilde{\ell}}(\Sigma)}^{\frac{1}{2}} \sup_{\|w\|_{L^{\infty}}\leq 1}\|\phi S_{|u|^{2\beta-1+ib}w}f\|_{\dot{H}^{-\tilde{\ell}}(\Sigma)}^{\frac12}+\|f\|_{L^p(\R^n)}\big)^{\theta}\|f\|_{L^p(\R^n)}^{1-\theta}\nonumber\\
    &\lesssim \|\phi \widehat{f}\|_{H^{\tilde{\ell}}(\Sigma)}^{\theta}\|f\|_{L^p(\R^n)}^{1-\theta}+\sup_{\|w\|_{L^{\infty}}\leq 1}\|\phi S_{|u|^{2\beta-1+ib}w}f\|_{\dot{H}^{-\tilde{\ell}}(\Sigma)}^{\theta}\|f\|_{L^p(\R^n)}^{1-\theta}+\|f\|_{L^p(\R^n)}.
\end{align}

We can obtain an estimate for $\sup_{\|w\|_{L^{\infty}}\leq 1}\|\phi S_{|u|^{\beta-1+ib}w}f\|_{\dot{H}^{-\tilde{\ell}}(\Sigma)}$ using \eqref{itereq}. Let $C$ be the implicit constant in \eqref{itereq}. If $\sup_{\|w\|_{L^{\infty}}\leq 1}\|\phi S_{|u|^{2\beta-1+ib}w}f\|_{\dot{H}^{-\tilde{\ell}}(\Sigma)}\leq (2C)^{\frac{1}{1-\theta}} \|f\|_{L^p(\R^n)}$, then we are done. If not, \textit{i.e.}  $\sup_{\|w\|_{L^{\infty}}\leq 1}\|\phi S_{|u|^{2\beta-1+ib}w}f\|_{\dot{H}^{-\tilde{\ell}}(\Sigma)} >(2C)^{\frac{1}{1-\theta}} \|f\|_{L^p(\R^n)}$, then \eqref{itereq} implies
\begin{align*}
     \sup_{\|w\|_{L^{\infty}}\leq 1}\|\phi S_{|u|^{2\beta-1+ib}w}f\|_{\dot{H}^{-\tilde{\ell}}(\Sigma)}^{\theta}&\leq C( \|\phi \widehat{f}\|_{H^{\tilde{\ell}}(\Sigma)}^{\theta}\|f\|_{L^p(\R^n)}^{1-\theta} +\|f\|_{L^p(\R^n)})(C\|f\|_{L^p(\R^n)}^{1-\theta})^{-1}\\
     &\leq \|\phi \widehat{f}\|_{H^{\tilde{\ell}}(\Sigma)}^{\theta}+\|f\|^{\theta}_{L^p(\R^n)}.
\end{align*}
In either case we have 
\begin{align} \label{beta-1}
    \sup_{\|w\|_{L^{\infty}}\leq 1}\|\phi S_{|u|^{2\beta-1+ib}w}f\|_{\dot{H}^{-\tilde{\ell}}(\Sigma)} \lesssim \|\phi\widehat{f}\|_{H^{\tilde{\ell}}(\Sigma)}+\|f\|_{L^{p}(\R^n)} 
\end{align} 
for all $b\in\R$, given that $f\in L^p(\R^n)$, $\kappa_p\geq 2\beta$, $p>1$ and $0 < \beta \leq 1.$

Plugging \eqref{beta-1} into \eqref{beta/2}, we have 
\begin{align} \label{beta/2-1}
\sup_{\|w\|_{L^{\infty}}\leq 1}\|\phi S_{|u|^{\beta-1+ib}w}f\|_{L^2(\Sigma)}&\lesssim \|\phi \widehat{f}\|_{H^{\tilde{\ell}}(\Sigma)}^{\frac12}(\|\phi \widehat{f}\|_{H^{\tilde{\ell}}(\Sigma)}+\|f\|_{L^p(\R^n)})^{\frac12}+\|f\|_{L^p(\R^n)}\nonumber\\
&\simeq \|\phi \widehat{f}\|_{H^{\tilde{\ell}}(\Sigma)}+\|f\|_{L^p(\R^n)} \\
&\leq \|\phi \widehat{f}\|_{H^{\ell}(\Sigma)}+\|f\|_{L^p(\R^n)}    
\end{align}
for any $\ell \geq \tilde{\ell} = \frac{\kappa_p \beta}{\kappa_p - \beta}$. This finishes the proof of Theorem \ref{main1} in the case $p > 1$.

\subsubsection{Case $p=1$}
Lemma \ref{apriori1} requires a bound from Proposition \ref{prop:bilinear} with $\gamma = \kappa_p < \frac{n-1}{2}$. This fails when $p=1$ and $\kappa_1 = \frac{n-1}{2}$. The following \textit{a priori} estimate can be used as a substitute.

\begin{lemma} \label{apriori2}
    Let $\varepsilon >0$. We have
    $$\sup_{\|w\|_{L^{\infty}(\R;\C)}\leq 1, b\in\R}\|\phi S_{|u|^{\kappa_1-1+ib}w(u)}f\|_{H^{-\kappa_1-\varepsilon}(\Sigma)}\lesssim \|f\|_{L^1(\R^n)}.$$
\end{lemma}
\begin{proof}
Fix $\varepsilon>0$. Note that for $f\in L^1(\R^n)$, we have 
\begin{align*}
    &\|\phi S_{|u|^{\kappa_1-1+ib}w(u)}f\|_{H^{-\kappa_1-\varepsilon}(\Sigma)}\\
    &= \bigg\|\phi(\xi)\int_{\R^n} W_{\kappa_1-1,b}(x_n)e^{-2\pi i (x'\xi+x_nh(\xi))}f(x',x_n)dx'dx_n\bigg\|_{H^{-\kappa_1-\varepsilon}(\xi\in \R^{n-1})} \\
    &\leq \int_{\R^n} W_{\kappa_1 - 1, b}(x_n)f(x', x_n) \big\|\phi(\xi) e^{-2\pi i(x'\cdot \xi + x_nh(\xi))}\big\|_{H^{-\kappa_1-\varepsilon}(\xi\in \R^{n-1})} dx' dx_n.
\end{align*}
With the assumption that $\nabla^2h \not= 0$ in the support of $\phi$, a standard stationary phase argument shows that the Fourier transform (in the $\xi$ variable) of $\phi(\xi) e^{-2\pi i(x'\cdot \xi + x_nh(\xi))}$ is bounded pointwise in $\R^{n-1}$ by $(1+|x_n|)^{-(n-1)/2}$. This is sufficient to conclude that $\big\|\phi(\xi) e^{-2\pi i(x'\cdot \xi + x_nh(\xi))}\big\|_{H^{-\kappa_1-\varepsilon}(\xi\in \R^{n-1})} \lesssim (1+|x_n|)^{-(n-1)/2}$. Then
\[
\|\phi S_{|u|^{\kappa_1-1+ib}w(u)}f\|_{H^{-\kappa_1-\varepsilon}(\Sigma)} \lesssim \int_{\R^n} \frac{|W_{\kappa_1-1,b}(x_n)|}{(1+|x_n|)^{(n-1)/2}} |f(x', x_n)| dx' dx_n \lesssim \|f\|_{L^1(\R^n)}
\]
since $|W_{\kappa_1-1, b}(x_n)| \lesssim |x_n|^{\kappa_1} = |x_n|^{(n-1)/2}$.
\end{proof}

The rest of the proof of Theorem~\ref{main1} proceeds similarly to the $p>1$ case. The proof of Lemma~\ref{iterate1} in the case $A=0$ shows that
\[
\|\phi S_{|u|^{\beta-1+ib}w}f\|_{L^2(\Sigma)}\lesssim \|\phi\widehat{f}\|_{H^{\ell}(\Sigma)}^{\frac12}\sup_{\|w\|_{L^{\infty}}\leq 1}\|\phi S_{|u|^{2\beta -1+ib}w}\|^{\frac12}_{H^{-\ell}(\Sigma)}+\|f\|_{L^p(\R^n)}
\]
with the non-homogeneous Sobolev norm $H^{-\ell}(\Sigma)$ on the 
right side.

If we set $\theta = \frac{\kappa_1-2\beta}{\kappa_1-\beta}$, which is in $[0,1]$ provided that $0\leq \beta \leq 1$ and $2\beta \leq \kappa_1$, and set $\tilde{\ell} = (1-\theta)(\kappa_1+\varepsilon) = \frac{(\kappa_1+\varepsilon) \beta}{\kappa_1 - \beta}$, then the interpolation steps proceed as follows.
\begin{align*}
    \|\phi S_{|u|^{2\beta -1+ib}w}f\|_{H^{-\tilde{\ell}}(\Sigma)}&\lesssim \sup_{\|w\|_{L^{\infty}}\leq 1}\|\phi S_{|u|^{\beta-1+ib}w}f\|_{L^2(\Sigma)}^{\theta}  \sup_{\|w\|_{L^{\infty}}\leq 1}\|\phi S_{|u|^{\kappa_1-1+ib}w}f\|_{H^{-\kappa_1-\varepsilon}(\Sigma)}^{1-\theta}\nonumber\\
    &\lesssim \big(\|\phi\widehat{f}\|_{H^{\tilde{\ell}}(\Sigma)}^{\frac{1}{2}} \sup_{\|w\|_{L^{\infty}}\leq 1}\|\phi S_{|u|^{2\beta-1+ib}w}f\|_{H^{-\tilde{\ell}}(\Sigma)}^{\frac12}+\|f\|_{L^p(\R^n)}\big)^{\theta}\|f\|_{L^p(\R^n)}^{1-\theta}\nonumber\\
    &\lesssim \|\phi \widehat{f}\|_{H^{\tilde{\ell}}(\Sigma)}^{\theta}\|f\|_{L^p(\R^n)}^{1-\theta}+\sup_{\|w\|_{L^{\infty}}\leq 1}\|\phi S_{|u|^{2\beta-1+ib}w}f\|_{H^{-\tilde{\ell}}(\Sigma)}^{\theta}\|f\|_{L^p(\R^n)}^{1-\theta}+\|f\|_{L^p(\R^n)}.
\end{align*}
From here one deduces that 
\[
\sup_{\|w\|_{L^{\infty}}\leq 1}\|\phi S_{|u|^{2\beta-1+ib}w}f\|_{H^{-\tilde{\ell}}(\Sigma)} \lesssim  \|\phi \widehat{f}\|_{H^{\tilde{\ell}}(\Sigma)} + \|f\|_{L^p(\R^n)}
\]
and finally
\[
\|\phi S_{|u|^{\beta-1+ib}w}f\|_{L^2(\Sigma)}\lesssim \|\phi \widehat{f}\|_{H^{\tilde{\ell}}(\Sigma)} + \|f\|_{L^p(\R^n)}.
\]
For any given $\ell > \frac{\kappa_1 \beta}{\kappa_1 - \beta}$, we may choose $\varepsilon > 0$ such that $\ell = \tilde{\ell} = \frac{(\kappa_1 + \varepsilon)\beta}{\kappa_1 - \beta}$. This completes the proof in the $p=1$ case.

\subsection{Higher power of $\beta$ (Proof of Theorem \ref{main2})}

Choose the value $\tilde{\ell} := \frac{\kappa_p(\beta - s)}{\kappa_p - \beta}$ so that the three points $(0, -\tilde{\ell})$, $(\beta, s)$, and $(\kappa_p, \kappa_p)$ are collinear in the plane.  Observe that for all admissible pairs $(\beta, s)$, the value $\tilde{\ell} = \frac{\kappa_p(\beta - s)}{\kappa_p - \beta} \leq \kappa_p$ since $\beta - s \leq \kappa_p - \beta$.

Suppose $0 \leq \tilde{\ell} \leq \kappa_p \leq 2$. Then Theorem \ref{main1} provides an estimate in $L^2(\Sigma)$,
\[
\sup_{\|w\|_{L^\infty(\R;\C)}\leq 1, b \in \R} \|\phi S_{|u|^{\tilde{\beta}-1+ib}w}f\|_{L^2(\Sigma)} \lesssim \|f\|_{L^p(\R^n)} + \|\phi \widehat{f}\|_{H^\ell(\Sigma)},
\]
where $\tilde{\beta} = \frac{\kappa_p \tilde{\ell}}{\kappa_p + \tilde{\ell}} = \frac{\kappa_p (\beta - s)}{\kappa_p - s}$ makes the point $(\tilde{\beta},0)$ collinear with the points named above. At the same time, Lemma \ref{apriori1} provides the estimate
\begin{equation} \label{kappa_p}
\sup_{\|w\|_{L^\infty(\R;\C)}\leq 1, b \in \R} \|\phi S_{|u|^{\kappa_p - 1 + ib}w}f\|_{\dot{H}^{-\kappa_p}(\Sigma)} \lesssim \|f\|_{L^p(\R^n)}.
\end{equation}
Complex interpolation now produces the desired result
\begin{equation*} 
\sup_{\|w\|_{L^\infty(\R;\C)}\leq 1, b \in \R} \|\phi S_{|u|^{\beta - 1 + ib}w}f\|_{\dot{H}^{-s}(\Sigma)} \lesssim \|f\|_{L^p(\R^n)},
\end{equation*}
as $s = (\frac{s}{\kappa_p}) \kappa_p +  (1-\frac{s}{\kappa_p})0$ while $\beta = (\frac{s}{\kappa_p}) \kappa_p + (1 - \frac{s}{\kappa_p}) \tilde{\beta}$.

The argument is identical when $\kappa_p > 2$ and $0 \leq \tilde{\ell} \leq \frac{\kappa_p}{\kappa_p - 1}$ so that $\tilde{\beta} \leq 1$ still holds.

Now suppose $\kappa_p > 2$ and $\frac{\kappa_p}{\kappa_p - 1} \leq \tilde{\ell} \leq 2$. We follow the structure of Theorem \ref{main1}, this time applying Lemma \ref{iterate1} to obtain
\[
\|\phi S_{|u|^{\tilde{\beta} - 1+ib}w}f\|_{\dot{H}^{-(\tilde{\beta}-1)}(\Sigma)}^2\lesssim \|\phi\widehat{f}\|_{H^{\tilde{\ell}}(\Sigma)}\sup_{\|w\|_{L^{\infty}}\leq 1} \|\phi S_{|u|^{2\tilde{\beta}-1+bi}w}f\|_{\dot{H}^{-2\tilde{\beta}+2-\tilde{\ell}}(\Sigma)}+\|f\|_{L^p(\R^n)}^2
\]
where $\tilde{\beta} = \frac{\kappa_p(\tilde{\ell}-1)}{\tilde{\ell}} = \kappa_p - \frac{\kappa_p - \beta}{\beta - s}$ is chosen so that the point $(\tilde{\beta}, \tilde{\beta}-1)$ is collinear with the points named above. Note that $2\tilde{\beta} \leq \kappa_p$ and $\tilde{\ell} = \frac{\kappa_p}{\kappa_p - \tilde{\beta}}$. Complex interpolation with \eqref{kappa_p} leads to the bounds
\begin{align*}
    &\|\phi S_{|u|^{2\tilde{\beta} -1+ib}w}f\|_{\dot{H}^{-2\tilde{\beta} + 2 -\tilde{\ell}}(\Sigma)}\\
    &\qquad \begin{aligned}[t]&\lesssim \sup_{\|w\|_{L^{\infty}}\leq 1}\|\phi S_{|u|^{\tilde{\beta}-1+ib}w}f\|_{\dot{H}^{-(\tilde{\beta}-1)}(\Sigma)}^{\theta}  \sup_{\|w\|_{L^{\infty}}\leq 1}\|\phi S_{|u|^{\kappa_p-1+ib}w}f\|_{\dot{H}^{-\kappa_p}(\Sigma)}^{1-\theta}\nonumber\\
    &\lesssim \big(\|\phi\widehat{f}\|_{H^{\tilde{\ell}}(\Sigma)}^{\frac{1}{2}} \sup_{\|w\|_{L^{\infty}}\leq 1}\|\phi S_{|u|^{2\tilde{\beta}-1+ib}w}f\|_{\dot{H}^{-2\tilde{\beta} + 2 -\tilde{\ell}}(\Sigma)}^{\frac12}+\|f\|_{L^p(\R^n)}\big)^{\theta}\|f\|_{L^p(\R^n)}^{1-\theta}\nonumber\\
    &\lesssim \|\phi \widehat{f}\|_{H^{\tilde{\ell}}(\Sigma)}^{\theta}\|f\|_{L^p(\R^n)}^{1-\theta}+\sup_{\|w\|_{L^{\infty}}\leq 1}\|\phi S_{|u|^{2\tilde{\beta}-1+ib}w}f\|_{\dot{H}^{-2\tilde{\beta} + 2 -\tilde{\ell}}(\Sigma)}^{\theta}\|f\|_{L^p(\R^n)}^{1-\theta}+\|f\|_{L^p(\R^n)}.
    \end{aligned}
\end{align*}
where $\theta = \frac{\kappa_p - 2\tilde{\beta}}{\kappa_p -\tilde{\beta}}= 2-\tilde{\ell}$. We have chosen $\theta$ so that $\theta\tilde{\beta} + (1-\theta)\kappa_p = 2\tilde{\beta}$. One can check that $\theta(\beta-1) + (1-\theta)\kappa_p = 2\tilde{\beta} - \theta = 2\tilde{\beta}-2+\tilde{\ell}$.

From here, it can be deduced that
\begin{equation*}
    \sup_{\|w\|_{L^{\infty}}\leq 1}\|\phi S_{|u|^{2\tilde{\beta}-1+ib}w}f\|_{\dot{H}^{-2\tilde{\beta} + 2 -\tilde{\ell}}(\Sigma)} \lesssim \|\phi\widehat{f}\|_{H^{\tilde{\ell}}(\Sigma)}+\|f\|_{L^{p}(\R^n)} 
\end{equation*} 
and then
\begin{equation} \label{tilde_beta}
    \sup_{\|w\|_{L^{\infty}}\leq 1}\|\phi S_{|u|^{\tilde{\beta}-1+ib}w}f\|_{\dot{H}^{-(\tilde{\beta}-1)}(\Sigma)} \lesssim \|\phi\widehat{f}\|_{H^{\tilde{\ell}}(\Sigma)}+\|f\|_{L^{p}(\R^n)}
\end{equation}
using the same argument as in Theorem \ref{main1}.

Interpolating between \eqref{tilde_beta} and \eqref{kappa_p}, this time with $\theta = \frac{\kappa_p - \beta}{\kappa_p - \tilde{\beta}} = \beta - s$, yields the desired bound
\begin{equation*}
    \sup_{\|w\|_{L^{\infty}}\leq 1}\|\phi S_{|u|^{\beta-1+ib}w}f\|_{\dot{H}^{-s}(\Sigma)} \lesssim \|\phi\widehat{f}\|_{H^{\tilde{\ell}}(\Sigma)}+\|f\|_{L^{p}(\R^n)}.
\end{equation*}

Lemma \ref{iterate1} becomes a significant bottleneck when $\tilde{\ell} > 2$, as the argument above requires  $(2\tilde{\beta}, 2\tilde{\beta}-2+\tilde{\ell})$ to be admissible. 

There is only one instance with $\tilde{\ell} > 2$ where admissibility takes place, almost by accident. This is the case $2 < \tilde{\ell} = \kappa_p < \frac{2\kappa_p}{2\kappa_p - \sigma_p}$. The fact that $2 < \tilde{\ell} = \kappa_p$ means that the pair $(\beta, s)$ lies on the line $2\beta - s = \kappa_p$ with $\kappa_p - 1 = \tilde{\beta}\leq \beta \leq \kappa_p$. Furthermore, $2\tilde{\beta} < \sigma_p$. Then Lemmas \ref{iterate1} and \ref{apriori1} combine to yield
\begin{align*}
\|\phi S_{|u|^{\kappa_p - 2+ib}w}f\|_{\dot{H}^{-(\kappa_p - 2)}(\Sigma)}^2 &\lesssim \|\phi\widehat{f}\|_{H^{\tilde{\ell}}(\Sigma)}\sup_{\|w\|_{L^{\infty}}\leq 1} \|\phi S_{|u|^{2\kappa_p - 3+bi}w}f\|_{\dot{H}^{-(3\kappa_p - 4)}(\Sigma)}+\|f\|_{L^p(\R^n)}^2 \\
&\lesssim \|\phi\widehat{f}\|_{H^{\tilde{\ell}}(\Sigma)}\|f\|_{L^p(\R^n)} +\|f\|_{L^p(\R^n)}^2,
\end{align*}
thus
\[
\|\phi S_{|u|^{\kappa_p - 2+ib}w}f\|_{\dot{H}^{-(\kappa_p - 2)}(\Sigma)} \lesssim \|\phi\widehat{f}\|_{H^{\tilde{\ell}}(\Sigma)} + \|f\|_{L^p(\R^n)}.
\]
Interpolating between this bound and \eqref{kappa_p} yields the desired result, since $\kappa_p-1 \leq \beta \leq \kappa_p$.

Now let us assume $2 < \tilde{\ell} < \min(\kappa_p, \frac{2\kappa_p}{2\kappa_p - \sigma_p})$.  Choose $\tilde{\beta} = \frac{\kappa_p(\tilde{\ell}-1)}{\tilde{\ell}}$ as before, so that $(\tilde{\beta}, \tilde{\beta}-1)$ is the first admissible point between $(0, -\tilde{\ell})$ and $(\kappa_p, \kappa_p)$. 

Unfortunately, Lemma \ref{iterate1} cannot be applied with an \textit{a priori} bound of $\| \phi \widehat{f}\|_{H^{\tilde{\ell}}(\Sigma)}$, as the constraint $4\tilde{\beta} - (2\tilde{\beta}-2+\tilde{\ell}) \leq \kappa_p$ is not satisfied.  As a substitute, we use Lemmas \ref{iterate1} and \eqref{apriori1} to extract the bound
\begin{align}
\|\phi S_{|u|^{\tilde{\beta} - 1+ib}w}f\|_{\dot{H}^{-(\tilde{\beta}-1)}(\Sigma)}^2&\lesssim \|\phi\widehat{f}\|_{H^{\kappa_p(1 - 2/\tilde{\ell})+2}(\Sigma)}\sup_{\|w\|_{L^{\infty}}\leq 1} \|\phi S_{|u|^{2\tilde{\beta}-1+bi}w}f\|_{\dot{H}^{-2\tilde{\beta}-\kappa_p(1-2/\tilde{\ell})}(\Sigma)}+\|f\|_{L^p(\R^n)}^2 \nonumber \\
&=\|\phi\widehat{f}\|_{H^{\kappa_p(1 - 2/\tilde{\ell})+2}(\Sigma)}\sup_{\|w\|_{L^{\infty}}\leq 1} \|\phi S_{|u|^{2\tilde{\beta}-1+bi}w}f\|_{\dot{H}^{-4\tilde{\beta}+\kappa_p}(\Sigma)}+\|f\|_{L^p(\R^n)}^2 \nonumber\\
&\lesssim \|\phi\widehat{f}\|_{H^{\kappa_p(1 - 2/\tilde{\ell})+2}(\Sigma)} \|f\|_{L^p(\R^n)} +\|f\|_{L^p(\R^n)}^2.
\end{align}
Interpolating between this bound and \eqref{kappa_p} yields the conclusion of Theorem \ref{main2}.

The proof for $p=1$ substitutes Lemma \ref{apriori2} as the available \textit{a priori} bound in place of \eqref{kappa_p}. Note that $\sigma_1 = \kappa_1$, so the cases $2 < \tilde{\ell} < \frac{2\kappa_1}{2\kappa_1 - \sigma_1}$ do not exist.

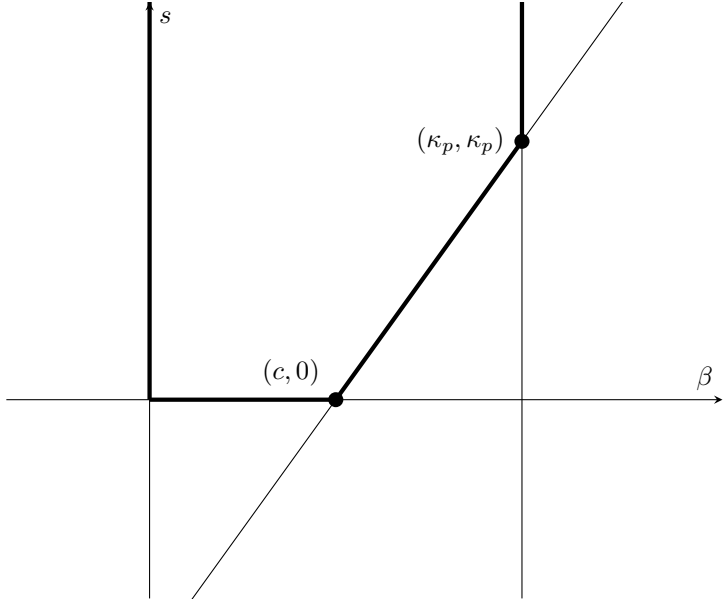
\begin{figure}
    \centering
    
\begin{tikzpicture}
\begin{axis}[
    axis lines = center,
    xlabel = \(\beta\),
    ylabel = {\(s\)},
    xmin=-0.5, xmax=2,
    ymin=-1, ymax=2,
    xtick={0},
    ytick={0},
    grid style=dashed,
    samples=100, 
]
\foreach \k in {1.3}{
\foreach \l in {1.6}{
\addplot[domain=0:4]{2*x-\k};
\addplot[mark=*] coordinates {(\k,\k)};
\addplot[mark=*] coordinates {(\k/2,0)};
\addplot [mark=none] coordinates {(\k, -1) (\k, 5)};
\node[label={180:{$(\kappa_p,\kappa_p)$}},circle,fill,inner sep=2pt] at (axis cs:1.3,1.3) {};
\node[label={145:{$(c,0)$}},circle,fill,inner sep=2pt] at (axis cs:1.3/2,0) {};
\addplot[ultra thick, , mark=none] coordinates {(0,0) (0,3)};
\addplot[ultra thick, , mark=none] coordinates {(\k,\k) (\k,3)};
\addplot[ultra thick, , mark=none] coordinates {(0,0) (1.3/2,0)};
\addplot[ultra thick, , mark=none] coordinates {(1.3/2,0) (\k,\k)};
}
}
\end{axis}
\end{tikzpicture}
    \caption{When $c= \min\lbrace \frac{\kappa_p\ell}{\kappa_p+\ell}, \frac{\kappa_p}{2}\rbrace\leq 1$, we directly interpolate the points $(c,0)$ and $(\kappa_p,\kappa_p)$. The region bounded by the bold lines is the region of all exponents that we can obtain, and it is optimal. The case $\kappa_p\leq 2$ and $\ell\geq 0$, the case $\kappa_p \geq 2$ and $\ell \in [0,1]$, and the case $\kappa_p\in [2, \frac{\ell}{\ell-1}]$ and $\ell \in (1,2]$ are in this situation.}
\end{figure}

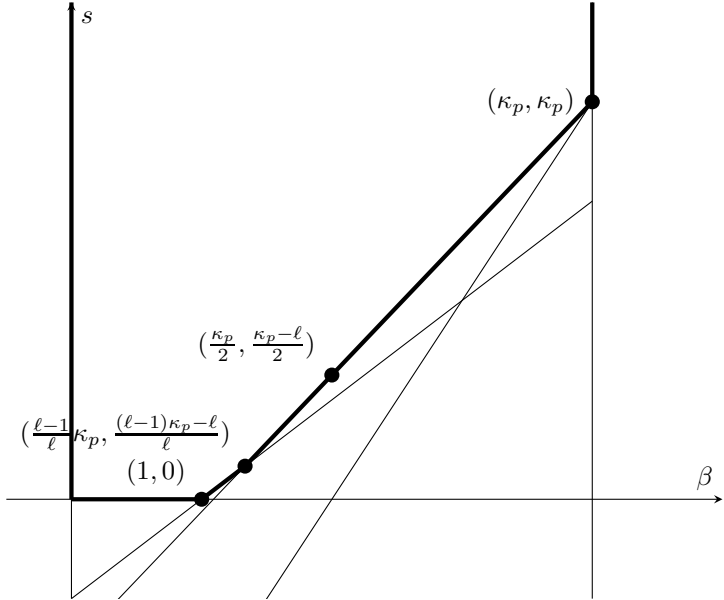
\begin{figure}
    \centering
    
\begin{tikzpicture}
\begin{axis}[
    axis lines = center,
    xlabel = \(\beta\),
    ylabel = {\(s\)},
    xmin=-0.5, xmax=5,
    ymin=-1, ymax=5,
    xtick={0},
    ytick={0},
    grid style=dashed,
    samples=100, 
]
\addplot[domain = 0:4]{x -1};
\foreach \k in {4}{
\foreach \l in {1.5}{
\addplot[domain=0:4]{2*x-\k};
\addplot[domain=0:4]{(\l * x + x * \k- \l * \k)/(\k)};
\addplot[mark=*] coordinates {(\k,\k)};
\node[label={145:{$(1,0)$}},circle,fill,inner sep=2pt] at (axis cs:1,0) {};
\addplot [mark=none] coordinates {(\k, -1) (\k, 5)};
\node[label={180:{$(\kappa_p,\kappa_p)$}},circle,fill,inner sep=2pt] at (axis cs:4,4) {};
\node[label={125:{$(\frac{\ell-1}{\ell}\kappa_p,\frac{(\ell-1)\kappa_p-\ell}{\ell})$}},circle,fill,inner sep=2pt] at (axis cs:4/3,1/3) {};
\node[label={145:{$(\frac{\kappa_p}{2},\frac{\kappa_p-\ell}{2})$}},circle,fill,inner sep=2pt] at (axis cs:2,5/4) {};
\addplot[ultra thick, , mark=none] coordinates {(0,0) (0,5)};
\addplot[ultra thick, , mark=none] coordinates {(\k,\k) (\k,5)};
\addplot[ultra thick, , mark=none] coordinates {(0,0) (1,0)};
\addplot[ultra thick, , mark=none] coordinates {(4/3,1/3) (1,0)};
\addplot[ultra thick, , mark=none] coordinates {(4/3,1/3) (\k,\k)};
}
}
\end{axis}
\end{tikzpicture}
    \caption{The case $\kappa_p\geq \frac{\ell}{\ell-1}$ and $\ell\in (1,2]$. The line joining $(1,0)$ and $(\frac{\ell-1}{\ell}\kappa_p,\frac{(\ell-1)\kappa_p-\ell}{\ell})$ is $s=\beta -1$, and the line joining $(\frac{\ell-1}{\ell}\kappa_p,\frac{(\ell-1)\kappa_p-\ell}{\ell})$ and $(\kappa_p,\kappa_p)$ is $s=\frac{\ell\beta -\ell\kappa_p+\kappa_p\beta}{\kappa_p}$. We can obtain the exponents between  $(\frac{\kappa_p}{2},\frac{(\kappa_p-\ell)}{2})$ and $(\kappa_p,\kappa_p)$, and an interpolation argument allows us to obtain the exponents between $(\frac{\ell-1}{\ell}\kappa_p, \frac{\ell-1}{\ell}\kappa_p-1)$ and $(\frac{\kappa_p}{2},\frac{(\kappa_p-\ell)}{2})$. The region bounded by the bold lines is the region of all exponents that we can obtain, and it is optimal. }
\end{figure}

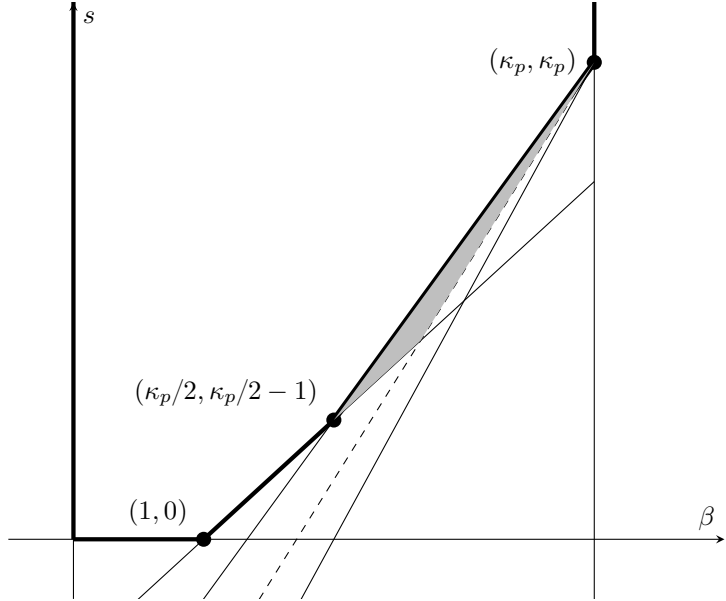
\begin{figure}
    \centering    
\begin{tikzpicture}
\begin{axis}[
    axis lines = center,
    xlabel = \(\beta\),
    ylabel = {\(s\)},
    xmin=-0.5, xmax=5,
    ymin=-0.5, ymax=4.5,
    xtick={0},
    ytick={0},
    grid style=dashed,
    samples=100, 
]
\addplot[domain = 0:4]{x -1};
\foreach \k in {4}{
\foreach \l in {3}{
\addplot[domain=0:4]{2*x-\k};
\addplot[domain=0:4,dashed]{(\l * x + x * \k- \l * \k)/(\k)};
\addplot[domain=0:4]{(2 * x + x * \k- 2 * \k)/(\k)};

\addplot[ultra thick, , mark=none] coordinates {(0,0) (0,5)};
\addplot[ultra thick, , mark=none] coordinates {(\k,\k) (\k,5)};
\addplot[ultra thick, , mark=none] coordinates {(0,0) (1,0)};
\addplot[ultra thick, , mark=none] coordinates {(\k/2,\k/2-1) (1,0)};
\addplot[line width=2.3pt, , mark=none] coordinates {(\k/2,\k/2-1) (\k,\k)};
\addplot[line width=2.3pt, , mark=none] coordinates {(\k/2+0.4,\k/2+0.4-1) (\k,\k)};
\addplot[ultra thick, , mark=none] coordinates {(\k/2+0.4,\k/2+0.4-1) (\k/2,\k/2-1)};
\node[label={145:{$(1,0)$}},circle,fill,inner sep=2pt] at (axis cs:1,0) {};
\addplot [mark=none] coordinates {(\k, -1) (\k, 5)};
\node[label={180:{$(\kappa_p,\kappa_p)$}},circle,fill,inner sep=2pt] at (axis cs:4,4) {};
\node[label={145:{$(\frac{\kappa_p}{2},\frac{\kappa_p}{2}-1)$}},circle,fill,inner sep=2pt] at (axis cs:4/2,1) {};
\node[label={315:{$(\frac{\kappa_p+\ell-2}{2},\frac{\kappa_p+\ell-2}{2}-1)$}},circle,fill,inner sep=2pt] at (axis cs:4/2+0.4,1.4) {};
\addplot[fill=lightgray, draw=none] coordinates {(\k/2+0.4,\k/2+0.4-1) (\k,\k) (8/3,5/3)};
\addplot[fill=darkgray, draw=none] coordinates {(\k/2+0.4,\k/2+0.4-1) (\k,\k) (\k/2,\k/2-1)};
}
}
\end{axis}
\end{tikzpicture}
    \caption{This figure shows a case of $\ell \geq 2$, $\kappa_p\geq 2$,  and $\kappa_p\geq \ell$. The line connecting $(1,0)$ and $(\kappa_p/2, \kappa_p/2-1)$ is $s=\beta -1$, the line joining $(\kappa_p/2, \kappa_p/2-1)$ and $(\kappa_p,\kappa_p)$ is $s=\frac{2\beta -2\kappa_p+\kappa_p\beta}{\kappa_p}$, the dashed line is $s=\frac{\ell\beta -\ell\kappa_p+\kappa_p\beta}{\kappa_p}$, and the line below the dashed line is $s=2\beta -\kappa_p$. The region bounded by the bold lines is the region of all possible exponents that we can obtain without any additional assumptions; the dark-gray shaded region is the region that we can obtain if we additionally assume that $\kappa_p<1+\frac{\sigma_p}{2}$; the light-gray shaded region is the region that is in the necessary conditions but we are not able to obtain.}
    \label{fig_f1}
\end{figure}

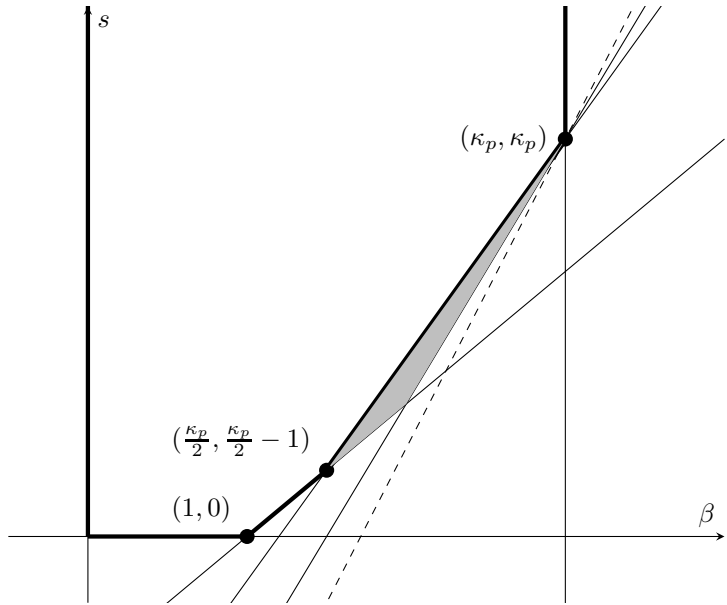
\begin{figure}
    \centering
    
\begin{tikzpicture}
\begin{axis}[
    axis lines = center,
    xlabel = \(\beta\),
    ylabel = {\(s\)},
    xmin=-0.5, xmax=4,
    ymin=-0.5, ymax=4,
    xtick={0},
    ytick={0},
    grid style=dashed,
    samples=100, 
]
\addplot[domain = 0:4]{x -1};
\foreach \k in {3}{
\foreach \l in {4}{
\addplot[domain=0:4]{2*x-\k};
\addplot[domain=0:4,dashed]{(\l * x + x * \k- \l * \k)/(\k)};
\addplot[domain=0:4]{(2 * x + x * \k- 2 * \k)/(\k)};
\addplot[ultra thick, , mark=none] coordinates {(0,0) (0,5)};
\addplot[ultra thick, , mark=none] coordinates {(\k,\k) (\k,5)};
\addplot[ultra thick, , mark=none] coordinates {(0,0) (1,0)};
\addplot[ultra thick, , mark=none] coordinates {(\k/2,\k/2-1) (1,0)};
\addplot[line width=2.3pt, , mark=none] coordinates {(\k/2,\k/2-1) (\k,\k)};
\addplot[mark=*] coordinates {(\k,\k)};
\addplot[mark=*] coordinates {(\k/2,\k/2-1)};
\node[label={145:{$(1,0)$}},circle,fill,inner sep=2pt] at (axis cs:1,0) {};
\addplot [mark=none] coordinates {(\k, -1) (\k, 5)};
\node[label={180:{$(\kappa_p,\kappa_p)$}},circle,fill,inner sep=2pt] at (axis cs:3,3) {};
\node[label={145:{$(\frac{\kappa_p}{2},\frac{\kappa_p}{2}-1)$}},circle,fill,inner sep=2pt] at (axis cs:3/2,1/2) {};
\addplot[fill=lightgray, draw=none] coordinates {(\k/2,\k/2-1) (\k,\k) (2,1)};
}
}
\end{axis}
\end{tikzpicture}
    \caption{This figure shows a case of $\ell \geq 2$, $\kappa_p\geq 2$, and $ \ell\geq \kappa_p$. The line connecting $(1,0)$ and $(\kappa_p/2, \kappa_p/2-1)$ is $s=\beta -1$, the line joining $(\kappa_p/2, \kappa_p/2-1)$ and $(\kappa_p,\kappa_p)$ is $s=\frac{2\beta -2\kappa_p+\kappa_p\beta}{\kappa_p}$, and the dashed line is $s=\frac{\ell\beta -\ell\kappa_p+\kappa_p\beta}{\kappa_p}$, which is below the line $s=2\beta -\kappa_p$. Although $s=\frac{\ell\beta -\ell\kappa_p+\kappa_p\beta}{\kappa_p}$ does not come into play, we cannot obtain the exponents in the shaded region because we are not able to obtain \eqref{HDR2} for $(\kappa_p-1, \kappa_p-2)$ due to the limitation of Lemma \ref{iterate1}.}
     \label{fig_f2}
\end{figure}

\section{Proof of Corollary \ref{coro1}}\label{proofcoro}
{ For every $f,g\in \mathcal{S}(\R^n)$, we have
\begin{align*}
    &\lan  T(f)(\xi,\eta) \phi(\xi), T(g)(\xi,\eta)\phi(\xi)\ran_{L^2_\xi L^2_\eta}  \\
    &=  
    \begin{aligned}[t]
    \iint_{\R^n\times\R^n} \Big\langle\frac{e^{-2\pi i x_n\eta}}{\eta}, \frac{e^{-2\pi i y_n\eta}}{\eta} \Big\rangle_{L^2_\eta} &\lan \phi e^{-2\pi i(x \cdot \xi)},\phi e^{-2\pi i(y \cdot \xi)}\ran_{L^2_{\xi}(\Sigma)} \\
    & \times  f(x',x_n) \overline{g(y',y_n)} dx' dx_n dy'dy_n
    \end{aligned} \\
    &= \iint_{\R^n\times \R^n} \lan s_r(x_n), s_r(y_n) \ran_{L^2_r} \widecheck{\phi^2}_{0,\Sigma}(x, y) f(x',x_n) \overline{g(y', y_n)} dx' dx_n dy' dy_n\\
    &= \pi \iint_{\R^n \times \R^n} K(x_n, y_n) \widecheck{\phi^2}_{0,\Sigma}(x, y) f(x',x_n)\overline{g(y',y_n)} dx' dx_n dy' dy_n,
\end{align*}
where $K(x_n, y_n) :=\chi_{[0,\infty)}(x_ny_n) \min\{|x_n|, |y_n|\}$.

Observe that
\[
\big| K(x_n, y_n) - \tfrac{|x_n|+ |y_n|}{2}\big| \leq \tfrac12|x_n - y_n|.
\]
This allows us to rewrite the inner product as
\begin{align*}
    &\lan  T(f)(\xi,\eta) \phi(\xi), T(g)(\xi,\eta)\phi(\xi)\ran_{L^2_\xi L^2_\eta} \\
    &= \pi \iint_{\R^n\times \R^n} \big[K(x_n, y_n) - \tfrac{|x_n| + |y_n|}{2} \big] \widecheck{\phi^2}_{0,\Sigma}(x, y) f(x',x_n)\overline{g(y',y_n)} dx'dx_n dy'dy_n \\
    & \quad + \pi \iint_{\R^n\times \R^n} \big[\tfrac{|x_n| + |y_n|}{2} \big] \widecheck{\phi^2}_{0,\Sigma}(x, y) f(x',x_n)\overline{g(y',y_n)} dx'dx_n dy'dy_n \\
    &=: (I) + (II).
\end{align*}

We first estimate $(I)$ using the observation above and Proposition \ref{prop:bilinear}.
\begin{align*}
    |(I)| &\lesssim \Big|\iint_{\R^n\times\R^n} |x_n - y_n|\widecheck{\phi^2}_{0,\Sigma}(x,y) f(x',x_n)\overline{g(y',y_n)}dx'dy' dx_n dy_n\Big| \\
    &\lesssim \iint_{\R\times\R} (1+ |x_n - y_n|)^{-(\frac{n-1}{p} - \frac{n+1}{2})} \|f( \cdot, x_n)\|_{L^p(\R^{n-1})} \|g( \cdot, y_n)\|_{L^p(\R^{n-1})} \\
    &\lesssim \|f\|_{L^p(\R^n)} \|g\|_{L^p(\R^n)}.
\end{align*}
The last step is the Hardy-Littlewood-Sobolev inequality, and it requires $1 \leq p \leq \frac{2n+2}{n+5}$ so that $\frac{n-1}{p} - \frac{n+1}{2} \geq 2 - \frac{2}{p}$. This is also equivalent to the condition $\kappa_p \geq 1$.

To estimate $(II)$, it suffices to consider 
$$\int_{\R^n\times \R^n} \frac{|x_n|}{2}\widecheck{\phi^2}_{0,\Sigma}(x,y)f(x)g(y)dxdy=\lan \phi \, \widehat{|x_n|f}, \phi \,\widehat{g}\ran_{L^2_{\xi}}.$$
By duality and applying Theorem \ref{submain2} with $\beta = 1$ and $s = \ell$, we have 
\begin{align*}
    |\lan \phi \, \widehat{|x_n|f}, \phi \,\widehat{g}\ran_{L^2_{\xi}}|&\lesssim \|\phi \widehat{|x_n|f}\|_{H^{-\ell}(\Sigma)} \|\phi \widehat{g}\|_{H^{\ell}(\Sigma)}\\
    &\lesssim (\|\phi \widehat{f}\|_{H^{\ell}(\Sigma)}+\|f\|_{L^p(\R^n)})(\|\phi \widehat{g}\|_{H^{\ell}(\Sigma)}+\|g\|_{L^p(\R^n)}).
\end{align*}
Since we may assume $\sigma_p > \kappa_p \geq 1$, these inequalities hold provided $\ell \geq \frac12$ and $\kappa_p \geq \max(2-\ell, \frac{\ell}{2\ell - 1}) = \frac{\ell}{2\ell - 1}$. That gives us the condition $\ell \geq \frac{\kappa_p}{2\kappa_p - 1}$.

Thus, we can conclude that 
\begin{align*}
    |\lan T(f)\phi, T(g)\phi\ran_{L^2_{\xi}L^2_{\eta}}|\lesssim (\|\phi \widehat{f}\|_{H^{\ell}(\Sigma)}+\|f\|_{L^p(\R^n)})(\|\phi{g}\|_{H^{\ell}(\Sigma)}+\|g\|_{L^p(\R^n)})
\end{align*}
for $1 < p \leq \frac{2n+2}{n+5}$ and $\ell \geq \frac{\kappa_p}{2\kappa_p - 1}$.  One can show the case $p=1$, $\ell > \frac{n-1}{2n-4}$ similarly.}

\section{Fractional-power difference quotients}\label{OE}
Let us consider the operator 
$$T_{\alpha}f(\xi,\eta)=\frac{\widehat{f}(\xi, h(\xi)+\eta)-\widehat{f}(\xi,h(\xi))}{|\eta|^{\alpha}}.$$
We can express $T_{\alpha}f$ in terms of $Tf$. More precisely, we have
\begin{align*}
    T_{\alpha}f(\xi,\eta) = \frac{\eta}{|\eta|^{\alpha}}Tf(\xi,\eta). 
\end{align*}

When $\alpha =1$, it follows immediately from Corollary \ref{coro1} that 
$$\|\phi T_{1}f\|_{L^2_{\xi}L^2_{\eta}}\lesssim\|f\|_{L^p(\R^n)}+\|\phi\widehat{f}\|_{H^{\ell}(\Sigma)}.$$

The following examples show that it is impossible to have a full $L^2_{\xi}L^2_{\eta}$ bound for $|\alpha-1|\geq \frac{1}{2}.$
\begin{eg}
Let $h(\xi)=|\xi|^2$. Consider $f(x)=e^{-|x|^2}$ with $\alpha \geq \frac{3}{2}$. We will show that it is impossible to obtain $L^2_{\xi}L^2_{\eta}$. Note that 
    $$T_{\alpha}f(\xi,\eta)= \frac{e^{-|\xi|^2-(|\xi|^2+\eta)^2}-e^{-|\xi|^2-|\xi|^4}}{|\eta|^{\alpha}}=e^{-|\xi|^2-|\xi|^4}\frac{e^{-2\eta|\xi|^2-\eta^2}-1}{|\eta|^{\alpha}}$$
    When $|\eta|<\frac{|\xi|}{1000}$, then
    \begin{align} \label{eg1}
        |T_{\alpha}f(\xi,\eta)|\geq e^{-|\xi|^2-|\xi|^4} \frac{c|\xi|^2e^{-c'|\xi|^4}}{|\eta|^{\alpha-1}} 
    \end{align}
    for some $c,c'>0$. When we square \eqref{eg1} both sides and integrate it with respect to $\eta \in (-\frac{|\xi|}{1000}, \frac{|\xi|}{1000})$, the integral will diverge because $2\alpha-2\in [1,\infty)$. 

    On the other hand, let us consider $\alpha \leq \frac{1}{2}$ with the same function $f$. If $\eta$ is sufficiently large so that $e^{-2\eta|\xi|^2-\eta^2} \leq \frac{1}{2}$, then 
    \begin{align} \label{eg2}
        |T_{\alpha}f(\xi,\eta)|\geq e^{-|\xi|^2-|\xi|^4} \frac{1}{2|\eta|^{\alpha}}
    \end{align}
    After squaring and integrating \eqref{eg2} with respect to $\eta$ in that range, this integral will diverge as $2\alpha \leq 1.$
\end{eg}

The difference quotient is locally square integrable in $\eta$ when $ \alpha < \frac12$. In fact, Bak, McMichael, and Oberlin have proved (as a consequence of~\cite{BMO}*{Theorem 4}) that
 \begin{equation}  \label{local L2}
 \int_{-1}^1 \int_U |\phi(\xi)T_\alpha f(\xi, \eta)|^2 d\xi d\eta \lesssim \|f\|_{L^p(\R^n)}^2 + \|\phi \widehat{f}\|_{L^2(\Sigma)}^2
 \end{equation}
 in the range $1 \leq p \leq \frac{2n+2}{n+4\alpha+1}$. For $p \leq \frac{2n+2}{n+3}$ the assumption on the $L^2$-trace of $\widehat{f}$ is automatically satisfied thanks to the Tomas-Stein restriction theorem. The endpoint case $\alpha = \frac12, p = \frac{2n+2}{n+3}$ is false because of Knapp counterexamples. 

Now we focus on the case $\alpha \in (\frac{1}{2},\frac{3}{2})$. Let us compute the kernel of $T_{\alpha}^*T_{\alpha}$ on $L^2_{\xi}L^2_{\tau}$.
$$\langle T_{\alpha}(f)(\xi,\eta)\phi(\xi), T_{\alpha}(g)(\xi,\eta) \phi(\xi)\rangle_{L^2_{\xi}L^2_{\eta}}=\int_{\R^n\times \R^n} K^{\alpha}_{0,0}(x,y)f(x)g(y)dxdy, $$
where 
$$K^{\alpha}_{0,0}(x,y) = I(x_n,y_n) \widecheck{\phi^2}_{0,\Sigma}(x,y)$$
and 
$$I(x_n,y_n) = \int_{\R} \frac{(e^{-2\pi i x_n\eta}-1)(e^{2\pi i y_n\eta}-1)}{|\eta|^{2\alpha}}d\eta.$$

By considering the principal value of $I(x_n,y_n)$ and using change of variables, 
\begin{align*}
    I(x_n,y_n)&=\int_{\R} \frac{\cos(2\pi (x_n-y_n)\eta)-\cos(2\pi y_n\eta)-\cos(2\pi x_n \eta)+1}{|\eta|^{2\alpha}}d\eta\\
    &= C_{\alpha}\big[|x_n|^{2\alpha-2}x_n +|y_n|^{2\alpha-2}y_n-|x_n-y_n|^{2\alpha-2}(x_n-y_n) \big],
\end{align*}
where $C_{\alpha}=\int_{\R}\frac{1-\cos(2\pi t)}{|t|^{2\alpha}}dt <\infty$ as $\alpha\in (\frac{1}{2},\frac{3}{2}).$

Therefore, we can write 
\begin{align*}
    &\langle T_{\alpha}(f)(\xi,\eta)\phi(\xi), T_{\alpha}(g)(\xi,\eta) \phi(\xi)\rangle_{L^2_{\xi}L^2_{\eta}} \\
    &=\int_{\R^n\times \R^n} C_{\alpha}\big[|x_n|^{2\alpha-2}x_n +|y_n|^{2\alpha-2}y_n]\widecheck{\phi^2}_{0,\Sigma}(x,y)f(x)g(y)dxdy \\
    &\quad - C_{\alpha}\int_{\R^n \times \R^n} [|x_n-y_n|^{2\alpha-2}(x_n-y_n)]\widecheck{\phi^2}_{0,\Sigma}(x,y)f(x)g(y)dxdy \\
    &=: (I)+(II).
\end{align*}
Using the Hardy-Littlewood-Sobolev inequality, for $1\leq p\leq \frac{2n+2}{n+4\alpha+1}$, we have 
$$|(II)|\lesssim \|f\|_{L^p(\R^n)} \|g\|_{L^p(\R^n)}.$$

To estimate $(I)$, it suffices to consider the $x_n$ case. Indeed, we can write 
\begin{align*}
    \int_{\R^n\times \R^n} |x_n|^{2\alpha-2}x_n \widecheck{\phi^2}_{0,\Sigma}(x,y)f(x)g(y)dydx &\simeq \langle \phi S_{|r|^{2\alpha-2}}f , \widehat{g} \phi \rangle_{L^2(\Sigma)}.
\end{align*}
Using Equation \eqref{beta-1} with $\beta = \alpha-\frac{1}{2} \in (\frac{1}{2},1)$, we have 
\begin{align*}
    \bigg| \int_{\R^n\times \R^n} |x_n|^{2\alpha-2}x_n \widecheck{\phi^2}_{0,\Sigma}(x,y) f(x)g(y)dydx \bigg| &\lesssim (\|f\|_{L^p(\R^n)}+\|\phi \widehat{f}\|_{H^{\ell}(\Sigma)})\|\phi \widehat{g}\|_{H^{\ell}(\Sigma)},
\end{align*}
for the case $\ell \geq \frac{\kappa_p(2\alpha-1)}{2\kappa_p-2\alpha+1}$ and $1<p\leq \frac{2n+2}{n+4\alpha+1}$ or the case $\ell >\frac{\kappa_p(2\alpha-1)}{2\kappa_p-2\alpha+1}$ and $p=1$. 

In summary, we have 
\begin{theorem} \label{Talphabdd}
     Let $\alpha\in (\frac{1}{2},\frac{3}{2})$. Suppose $f\in L^p(\R^n)$ has $\widehat{f}|_{\Sigma}\in H^{\ell}(\Sigma)$, and $\phi\in C^{\infty}_c(\Sigma)$. If $\ell \geq \frac{\kappa_p(2\alpha-1)}{2\kappa_p-2\alpha+1}$ and $1<p\leq \frac{2n+2}{n+4\alpha+1}$, then 
$$\|\phi T_{\alpha}f\|_{L^2_{\xi}L^2_{\eta}}\lesssim \|f\|_{L^p(\R^n)}+\|\phi \widehat{f}\|_{H^{\ell}(\Sigma)}.$$
The estimate is also true for $p=1$ if $\ell>\frac{\kappa_1(2\alpha-1)}{2\kappa_1-2\alpha+1}=\frac{(n-1)(2\alpha-1)}{2(n-2\alpha)}.$
\end{theorem}

One can avoid the immediate local integrability concerns in \eqref{eg1} when $\frac32 \leq \alpha < \frac52$ by considering the second-order difference quotient $\frac12(T_\alpha f(\xi, \eta) + T_\alpha f(\xi, -\eta))$. Nevertheless, estimates of the type
\[
\|\phi(T_\alpha f(\xi, \eta) + T_\alpha f(\xi, -\eta))\|_{L^2_\xi L^2_\eta} \lesssim \|f\|_{L^p(\R^n)} + \|\phi\widehat{f}\|_{H^\ell(\Sigma)}
\]
are still false for $\alpha \geq \frac32$ regardless of the size of $\ell$.

To construct a counterexample, start with a function $f \in \mathcal{S}(\R^n)$ which has $\widehat{f}(\xi, h(\xi)) \equiv 0$ and $\frac{\partial\widehat{f}}{\partial \xi_n}(\xi, h(\xi)) \not\equiv 0$ on $\Sigma$. Then consider $f(x', x_n - M)$ for a large number $M$. Its Fourier transform is $\widehat{f}(\xi, \xi_n)e^{-i M \xi_n}$, which also vanishes on $\Sigma$.

Then the leading-order behavior for small $\eta$ is
\[
\phi(\xi)\frac{T_\alpha f(\xi, \eta) + T_\alpha f(\xi, -\eta)}{2|\eta|^{\alpha}} \sim \phi(\xi)\tfrac{\partial \widehat{f}}{\partial \xi_n}(\xi, h(\xi)) \frac{\eta \sin(M \eta)}{|\eta|^\alpha},
\]
and the integral in the $\eta$ direction yields
\[
\int_{-1}^1 \frac{\sin^2( M \eta)}{|\eta|^{2\alpha-2}}d\eta \gtrsim 
\begin{cases} M^{2\alpha-3}, &\text{ if } \alpha > \frac32 \\
\log(1+M), &\text{ if } \alpha = \frac32
\end{cases}
\]
which grows without bound as $M \to \infty$.

Local integrability (as in \eqref{local L2}) remains an open question in the endpoint cases $\alpha = \frac12$, $p = \frac{2n+2}{n+3}$, for any $\ell > 0$.

\section*{Acknowledgement}
Michael Goldberg received support from Simons Foundation grant \#635369.

\Addresses
\end{document}